\documentclass[12pt]{amsart}

\pdfoutput=1
\usepackage{amscd,amssymb,amsmath,amsfonts,amsthm}
\usepackage{rotating,graphicx,url}
\usepackage[matrix,graph]{xy}
\usepackage{hyperref}

\theoremstyle{plain}
\newtheorem{thm}{Theorem}[section]
\newtheorem*{thm*}{Theorem}
\newtheorem{lem}[thm]{Lemma}
\newtheorem{prop}[thm]{Proposition}
\newtheorem{cor}[thm]{Corollary}
\newtheorem{conj}[thm]{Conjecture}

\theoremstyle{definition}
\newtheorem{defn}[thm]{Definition}

\newtheorem{ex}[thm]{Example}

\theoremstyle{remark}
\newtheorem{rmk}[thm]{Remark}

\newcommand{\CP}{\mathbb{P}}

\newcommand{\C}{\operatorname{\mathbb{C}}}

\newcommand{\pt}{{\{ \mbox{pt.} \}}}

\newcommand{\Span}[1]{\langle #1 \rangle}
\newcommand{\Sym}{\operatorname{Sym}}

\newcommand{\into}{\hookrightarrow}

\newcommand{\iso}{{\widetilde\longrightarrow}}

\DeclareMathOperator{\Hom}{Hom}

\newcommand{\suml}{\sum\limits}
\newcommand{\prodl}{\prod\limits}

\renewcommand{\dim}{\operatorname{\rm dim}}
\newcommand{\codim}{\operatorname{\rm codim}}
\newcommand{\supp}{\operatorname{\sf supp}}

\newcommand{\id}{\rm id}

\newcommand{\rk}{\operatorname{{\rm rk}}}

\newcommand{\sH}{\mathcal{H}}  
\newcommand{\VD}{\mathcal{D}}  

\newcommand{\Chi}{\ensuremath{ \raisebox{0.4ex}{$\chi$} }}

\newcommand{\abs}[1]{\lvert #1 \rvert}

\newcommand{\csm}{c_{\text{\rm SM}}}
\newcommand{\cm}{c_{\text{\rm M}}}
\newcommand{\CSM}{\csm}
\newcommand{\CM}{\cm}
\newcommand{\XT}{\widetilde{X}}
\newcommand{\cS}{\mathcal{S}}

\newcommand{\GL}{\operatorname{GL}}

\newcommand{\GR}[2]{Gr(#1, \C^{#2})}
\newcommand{\GRb}[2]{Gr(#1, #2)}
\newcommand{\Grass}{\prod_{i=1}^m \GRb{k_i}{V}}

\renewcommand{\P}[2]{P_{#1, #2}}
\newcommand{\Dual}{\operatorname{\mathbb{D}}}
\newcommand{\TDual}{\operatorname{\widetilde{\mathbb{D}}}}

\newcommand{\dab}{d_{\alpha, \beta}}
\newcommand{\xa}{{X_\alpha}}
\newcommand{\xac}{{X^\circ_\alpha}}
\newcommand{\za}{{Z_\alpha}}

\newcommand{\xb}{{X_\beta}}
\newcommand{\xbc}{{X^\circ_\beta}}
\newcommand{\zb}{{Z_\beta}}
\newcommand{\pbs}{(\pi_\beta)_*}
\newcommand{\pb}{\pi_\beta}
\newcommand{\eab}{e_{\alpha, \beta}}

\newcommand{\pas}{(\pi_\alpha)_*}
\newcommand{\pa}{\pi_\alpha}

\newcommand{\zap}{Z_{\alpha, s}}
\newcommand{\pap}{\pi_{\alpha, s}}

\newcommand{\fl}{\operatorname{Fl}}

\newcommand{\pr}{\operatorname{pr}}

\newcommand{\leftclass}[1]{ s_{r_{#1} \times (k_{#1}-l_{#1})}( \bU_{#1-1}/\bU_{#1} ) }

\newcommand{\V}{V}

\newcommand{\sleft}{\mbox{\tiny left}}
\newcommand{\sright}{\mbox{\tiny right}}

\newcommand{\Path}[2]{\xymatrix{ #1 \ar@{-->}[r] & #2}}

\newcommand{\cont}{\ar@{>->}}

\newcommand{\bF}{\underline{F}}
\newcommand{\bC}{\underline{\C}}
\newcommand{\bU}{\underline{U}}
\newcommand{\bS}{\underline{S}}
\newcommand{\bQ}{\underline{Q}}
\newcommand{\bE}{\underline{E}}
\newcommand{\bA}{\underline{A}}
\newcommand{\bB}{\underline{B}}

\newcommand{\bV}{\underline{V}}
\newcommand{\blW}{\underline{W}^{\mbox{\tiny \sleft}}}
\newcommand{\brW}{\underline{W}^{\mbox{\tiny \sright}}}
\newcommand{\lW}{{W}^{\mbox{\tiny \sleft}}}
\newcommand{\rW}{{W}^{\mbox{\tiny \sright}}}

\newcommand{\sF}{\mathcal{F}^\bullet}

\newcommand{\IC}[1]{\mathcal{IC}^\bullet_{#1}}

\newcommand{\one}{1\hskip-4pt1}

\newcommand{\tCSM}{Chern-Schwartz-MacPherson }
\newcommand{\tCM}{Chern-Mather }

\begin{document}

\title{SINGULAR CHERN CLASSES OF SCHUBERT \\
VARIETIES VIA SMALL RESOLUTION}            

\author[Benjamin F Jones]{Benjamin F Jones}          
\address{Department of Mathematics, University of Georgia, Athens, GA 30602-3704}
\email{bjones@math.uga.edu}

\begin{abstract}
  We discuss a method for calculating the \tCSM (CSM) class of a
  Schubert variety in the Grassmannian using small resolutions
  introduced by Zelevinsky. As a consequence, we show how to compute
  the \tCM class and local Euler obstructions using small resolutions
  instead of the Nash blowup. The algorithm obtained for CSM classes
  also allows us to prove new cases of a positivity conjecture of
  Aluffi and Mihalcea. 
\end{abstract}

\maketitle

\tableofcontents

\section{Introduction}
\label{sec:intro}

\subsection{Schubert varieties and singular Chern classes}
\label{subsec:intro-schub-and-csm}

In this paper we consider Schubert varieties in the Grassmannian and
their singular Chern classes. A Schubert variety in the Grassmannian
$\GRb{k}{V}$ parametrizes the $k$-dimensional subspaces of $V \cong \C^n$ which satisfy certain intersection conditions with a
fixed complete flag in $V$.  These varieties play an important role in
many areas of mathematics from intersection theory and enumerative
geometry to representation theory and combinatorics. The singular
Chern classes we consider are the \tCSM (CSM) class, constructed by
MacPherson in \cite{MacP}, and the \tCM class. Each of these classes
is an element of the Chow group of a complex algebraic variety. For a
non-singular variety, the Chern classes considered coincide with the
total (homology) Chern class of the tangent bundle.

Recently, the study of CSM classes of Schubert varieties in the
Grassmannian was initiated by Aluffi and Mihalcea in \cite{AM}.  In
particular, the explicit calculation of coefficients in the Schubert
basis of CSM classes of Schubert cells is completed in \cite{AM}. The
motivation for the present paper was to study the questions raised in
\cite{AM} from the point of view of small resolutions of
singularities. In particular, we answer the questions of computing the
Chern-Mather classes and local Euler obstructions for Schubert
varieties. The latter turn out to coincide with invariants coming from 
intersection cohomology. 

The main idea of our approach to understanding singular Chern classes is to use 
knowledge about explicit resolutions combined with the good functorial behavior of CSM classes. In particular, when small resolutions exist, the intimate relationship between the geometry and topology of the variety and its resolution lead to results relating the (singular) Chern classes. For example, an interesting consequence of the approach
above is the identification of the \tCM class and the stringy Chern class (cf. \cite{deFernex, Alu05}) in certain cases, see Remark \ref{rmk:stringy}.

The methods presented here provide a different approach than exists in the literature to
understanding and explicitly computing CSM classes. The approach, especially the existence of multiple 
distinct (Zelevinsky) resolutions, is useful for deriving pushforward formulas and exploring the positivity conjecture
of Aluffi and Mihalcae in some new cases.

\subsection{Statement of results}
\label{subsec:statement-of-results}

The first part of the paper deals with computing \tCSM and \tCM classes for general varieties using 
the tangent bundles of resolutions of singularities. In \S\ref{subsec:csm-by-resolution} we use the functorial properties of the \tCSM class to derive a formula for the 
class of a given complex variety. When one has explicit information about the resolution
of singularities involved, this approach leads to explicit computations. We 
carry out these computations in the case of Schubert cells and varieties $X$ in the Grassmannian in \S\ref{subsec:explicit}.


The first main result is Theorem \ref{thm:cm-xa}. We
prove that for certain varieties $X$, the \tCM class of $X$ coincides with the
pushforward of the CSM class of any small resolution. The Theorem allows us to 
compute the \tCM class without having to understand the Nash Blowup of $X$ explicitly.

\begin{thm*} 
	\label{thm:main-thm-3} 
	Let $X$ be a
complex algebraic variety which admits a small resolution and whose 
characteristic cycle is irreducible.
	Then, the \tCM class of a
  $X$ equals the push-forward of the Chern
  class of any small resolution $\pi : Z \to X$:
  \[ \CM( X ) = (\pi)_*( c(TZ) \cap [Z] ) .\]
\end{thm*}

The class on the right hand side coincides with the CSM (and the CM)
class of $Z$ since $Z$ is non-singular.  The theorem follows by
identifying the local Euler obstruction with an invariant coming from
the theory of intersection cohomology and using a theorem of Dubson and 
Kashiwara \cite{Dubson:index, Kash}.
The Schubert varieties in a Grassmannian satisfy the hypotheses above. 
In \S\ref{subsec:cm-computation} we compute the \tCM classes of all Schubert varieties in
$\GR{3}{6}$.  The identification of this class and the explicit calculations seem
to be new. The form of the identification also suggests that the
geometry of the Nash blowup of a Schubert variety is closely related
to the geometry of its small resolutions.


In order to carry out calculations in the context of the Grassmannian,
we need to understand the tangent bundle on a resolution of singularities of
a Schubert variety. 
We focus on the resolutions defined by Zelevinsky in \cite{Zel}.
The second main result is a formula for the CSM
class of a Zelevinsky resolution. Such a resolution $Z$ is
defined as a subvariety of a product of Grassmannians \[ Z \subset
\prodl_{i=1}^m \GRb{k_i}{V} = X.\] The universal sub-bundle $\bU_i$ on
each factor $\GRb{k_i}{V}$ pulls back to $X$ and also $Z$. The
incidence relations defining $Z$ as a subvariety of $X$ give rise to 
corresponding incidence
relations involving certain vector bundles on $Z$: 
\[ \blW_i \subset \bU_i \subset \brW_i \quad \text{for} \quad i=1,
\ldots, m\] (see
\S~\ref{subsec:zel-reso} for definitions). The following is a
restatement of Theorem \ref{thm:c-tangent-bundle}.
\begin{thm*}
\label{thm:main-thm-1}
  The CSM class of $Z$ is
  \[ \csm(Z) = \left( \prodl_{i=1}^m c( (\bU_i/\blW_i)^\vee \otimes
  (\brW_i/\bU_i) ) \right) \cap [Z] .\]
\end{thm*}

In section \ref{sec:csm-reso} we discuss the pushforward of the class
$\csm(Z) = c(TZ) \cap [Z]$ to the ambient Grassmannian of the Schubert
variety being resolved. This computes the main ingredient of
our formula. We give two methods for
computing the pushforward. The first method uses the localization
theorem for equivariant Chow groups. We present calculations of the
CSM classes of all Schubert cells in $\GR{3}{6}$ using this
method. The second method, described in \S\ref{subsec:direct-pushforward}, is based on an explicit description of the
Gysin map for Grassmannian bundles given by Fulton and Pragacz and exploits the
combinatorial simplicity of certain Zelevinsky resolutions which exist for any Schubert variety. 


In the last section we explore the positivity conjecture of Aluffi and Mihalcae from \cite{AM}. The conjecture states that the CSM class of any Schubert cell is positive. This amounts to the statement that the coefficients of the classes of Schubert varieties contained in the boundary of the closure of a given cell are non-negative. By additivity of the CSM class, the conjecture implies a weaker positivity conjecture concerning the CSM classes of Schubert varieties. We prove this weaker version in an infinite family of special cases in \S\ref{subsec:weak_conj}. For Schubert cells, proving positivity is more difficult. In \S\ref{subsec:codimension-one}, we show how the methods presented here may be used to prove that coefficients with small codimension in the CSM class of an arbitrary Schubert cell are non-negative. In particular, we give a new combinatorial interpretation of codimension one coefficients.


\section{CSM Classes and Resolution of Singularities}
\label{sec:csm}

\subsection{Chern-Schwartz-MacPherson Classes}
\label{subsec:csm-defn}

In \cite{MacP}, MacPherson defines two classes for each complex
algebraic variety $X$: the \tCM (CM) class, denoted $\cm(X)$, and the
\tCSM (CSM) class, denoted $\csm(X)$.  Both classes satisfy the
property that if $X$ is smooth, then the class coincides with $c(TX)
\cap [X]$, the total homology Chern class of $X$. The task then is to
define a class for varieties which are not smooth. In this case the
tangent bundle $TX$ is not available. We briefly recall how the CM and
CSM classes are defined and refer the reader to \cite{MacP} for
details.

Let $X$ be an irreducible complex algebraic variety and $\nu : \XT \to X$ be the Nash
blowup with tautological Nash bundle $\bS$. The class $\cm(X) :=
\nu_*( c(\bS) \cap [\XT] ) \in A_* X$ is called the \emph{\tCM class}
of $X$. It follows from properties of the Nash blowup that $\cm(X) \in [X] +
\sum_{i<\dim X}A_i X$.

On the other hand, the CSM class is defined in terms of the CM class
and an invariant called the local Euler obstruction which we define.
Let $p \in X$ and let $S(\nu^{-1}(p), \XT )$ denote the Segre class of the
normal cone to $\nu^{-1}(p)$ in $\XT$ (see \cite{Ful}, \S~4.2).  The
local Euler obstruction of $X$ at $p$ is
\begin{equation} \label{eq:euler-obs}
 Eu_X(p) := \int_{\nu^{-1}(p)} c( \bS\lvert_{\nu^{-1}(p)} ) \cap
S(\nu^{-1}(p), \XT )
\end{equation}
This formula is due to Gonzalez-Sprinberg and Verdier, MacPherson's
original definition is topological. The function $Eu_X : X \to
\mathbb{Z}$ is constructible with respect to any Whitney
stratification of $X$ (cf. \cite{MacP}, Lemma 2).

The integers defined by \eqref{eq:euler-obs}
determine a group automorphism of $F_* X$, the group of constructible
functions on $X$, as follows. For $W \subset X$ a closed subset, let
$\one_W$ denote the characteristic function of $W$. Then, 
$$Eu(\one_W)(p) = \left\{ \begin{array}{ll} Eu_W(p) & \text{if} \,\, p
    \in W \\ 0 & \text{otherwise} \end{array} \right. .$$ This extends
by linearity to an automorphism of $F_* X$, see \cite[Lemma 2]{MacP}
and \cite[Proposition \S6]{Gonz}.

Now let $X = \cup_{i \in I} S_i$ be a finite Whitney stratification of $X$
into smooth, locally closed subsets and fix a total ordering on $I$
compatible with the partial order given by inclusion of
closures of the $S_i$. The CSM class of the closure of a stratum is
defined below. This definition suffices to define the CSM class of
$X$ since it is the closure of the unique open stratum.

By \cite[Proposition \S6]{Gonz} there are integers $f_{i,j}$ such
that $$ Eu^{-1}(\one_{S_i}) = \suml_j f_{i,j} \one_{\overline{S_j}}
.$$ Then $f_{i,j}$ are entries of the matrix which is inverse to
the matrix with entries $e_{i,j} = Eu_{\overline{S_i}}(p_j)$ (for arbitrary points $p_j \in S_j$). We remark that the matrix with
entries $f_{i,j}$ ($e_{i,j}$ respectively) is upper triangular with
$1$'s along the diagonal. Upper triangularity follows from the
definition and the statement about the diagonal follows from the fact
that $Eu_X(p) = 1$ whenever $p$ is a smooth point of $X$
(cf. \cite{MacP}, p.426).

The following map defines a morphism from the group of functions which are
constructible with respect to the stratification $\cup_i S_i$ to the
Chow group of $X$:
\begin{equation} \label{eq:csm-formula-macp}
	c_*( \one_{\overline{S_i}} ) = \suml_j f_{i,j} \cdot \cm( \overline{S_j} ) .
\end{equation}
This suffices to define a map $c_* : F_* X \to A_* X$ since there always exists a stratification of $X$ which includes a given closed
subset as the closure of some stratum. We have defined $c_*$
to take values in the Chow group of $X$ rather than homology, as is done in
\cite{MacP}. See \cite{Ful}, \S 19.1.7 for comments about this refinement.

For $f: X \to Y$ a proper morphism, there is an induced pushforward
map $f_* : F_* X \to F_* Y$. Let $W \subset X$ be locally closed and $y
\in Y$, then $f_*(\one_W)(y) = \Chi(f^{-1}(y) \cap W)$ where $\Chi$ denotes
the topological Euler characteristic.

The following theorem is the main result of \cite{MacP} stated in
terms of the Chow group.

\begin{thm} \label{thm:macp-thm}
(MacPherson) The transformation $c_* \colon
  F_* \to A_*$ is natural. That is, if $f \colon Y
  \to X$ is a proper map, the following diagram commutes:
  \[ \xymatrix{
    F_* Y \ar[r]^{c_*} \ar[d]_{f_*} & A_* Y \ar[d]^{f_*} \\
    F_* X \ar[r]^{c_*} & A_* X .} \]
\end{thm}

For $W \subset X$ closed, we set the \emph{\tCSM class} of $W$ to be 
\begin{equation} \label{eq:csm-defn}
\csm(W) := c_*( \one_W ) .
\end{equation}
Since the matrix with entries $f_{i,j}$ is upper triangular with $1$'s
along the diagonal and $\cm(X) \in [X] + \sum_{i<\dim X}A_i X$ it
follows that $\csm(X) \in [X] + \sum_{i<\dim X}A_i X$ as well. 

Theorem \ref{thm:macp-thm} allows us to define a CSM class
for each locally closed subset of a fixed ambient variety. Let $S
\subset M$ be locally closed, $S = X \setminus Y$ with $X, Y$
closed in $M$. Then we set
\begin{equation} \label{eq:csm-cell-defn}
\csm(S) = \csm(X) - \csm(Y)
\end{equation}
and regard the class $\csm(S)$ as an element of $A_* M$.  Theorem
\ref{thm:macp-thm} implies that $\csm(S)$, as defined above, is
independent of the choice of $X$ and $Y$. However, in general it
depends on the embedding $S \into M$. If $i : M \to N$ is a proper
morphism and $i_*$ is injective (e.g. if $N$ has a cellular
decomposition and $M$ is the closure of a cell in $N$), we abuse
notation and write $\csm(S)$ for both the class in $A_* M$ and the
class $i_* \csm(S) \in A_* N$.

\subsection{Chern Classes by Resolution of Singularities}
\label{subsec:csm-by-resolution}

The \tCM class is difficult to calculate from the definition because
the Nash blowup is poorly understood in general. Also, its construction
is not functorial in the sense of Theorem \ref{thm:macp-thm}. This is a
major advantage to studying the CSM class. We will use
functoriality along with resolution of singularities to calculate
CSM classes. As a consequence of this
approach we obtain an effective algorithm for calculating the \tCM
class for varieties satisfying hypotheses on their resolutions and characteristic cycles. We now describe how to compute the CSM class using a family of
resolutions. In this paper a \emph{resolution of singularities} means a map of
algebraic varieties $\pi : Z \to X$ which is proper and birational with
$Z$ non-singular.

Suppose we embed a variety $X$ into a smooth variety $M$ which is
stratified compatibly with $X$. Let $M = \cup_{i \in I} S_i$ with $X =
\overline{S_{i_0}}$ for some index $i_0 \in I$. Also, let the index set $I$
have a total ordering such that $i \le j$ if $S_i \subset
\overline{S_j}$. Suppose there exist resolutions of
singularities  $Z_i \overset{\pi_i}{\longrightarrow} \overline{S_i}$
(for each $i \in I$) which are compatible with the stratification of
$M$ in the sense that the restriction of $\pi_i$ over any
stratum $S_j$ with $S_j \subset \overline{S_i}$ is a fiber bundle in the complex topology.

For $i, j \in I$ define $d_{i, j} := \Chi( \pi_i^{-1}(p) )$ for any $p
\in S_j$.  Naturality of the transformation $c_*$ implies
\begin{equation} \label{eq:pistar-CSM-Z} 
	(\pi_i)_* \CSM(Z_i) = (\pi_i)_* c_* (\one_{Z_i}) = c_* (\pi_i)_*
(\one_{Z_i}) = \sum_{j \le i} d_{i, j} c_* (\one_{S_j}) .
\end{equation}
The last equality follows from the definition of pushforward for
constructible functions, see \cite{MacP}.

We may view the system of
equations above as a matrix equation. The matrix with entries $d_{i,j}$
is an upper triangular, integer matrix with 1's on the diagonal. Hence
it is invertible. Solving for $\CSM(S_i)$ yields:
\begin{equation}
  \label{eq:csm-formula-general}
  \CSM(S_i) = \sum_{j \le i} e_{i, j} \cdot (\pi_j)_* \CSM(Z_j) .
\end{equation}
where the integers $e_{i,j}$ are entries of the matrix which is
inverse to the matrix with entries $d_{i,j}$.

By additivity, the CSM class of $X = \overline{S_{i_0}}$ is
the sum of the CSM classes of strata contained in $X$:
\begin{equation} \label{eq:csm-formula-X}
\csm(X) = \sum_{S_i \subset X} \csm(S_i)
\end{equation}
Since the $Z_i$ are smooth, the previous two
formulae give a concrete method of calculating CSM classes once the
tangent bundles of the resolutions and the Euler characteristics of
their fibers are understood. 


\section{The \tCM Class }
\label{sec:micro-local}

After recalling some basic results from microlocal geometry, we show how the approach to understanding \tCSM classes described in the Section \ref{sec:csm}
 allows us to identify (and explicitly calculate) both
MacPherson's local Euler obstructions and the \tCM classes of certain complex algebraic varieties including the Schubert varieties in a Grassmannian.

\subsection{Perverse Sheaves} 
\label{subsec:perverse-sheaves}

Let $X \subset Y$ be irreducible varieties.
Corresponding to $X$ there is a complex of 
sheaves of $\C$-vector spaces on $Y$ called the 
Deligne-Goresky-MacPherson sheaf (or intersection cohomology sheaf) 
of $X$. We denote this sheaf by $\IC{X}$. It is constructible with respect to any Whitney stratification of $Y$.

The sheaf $\IC{X}$ is a perverse sheaf on $Y$
with respect to the middle perversity function $p : 2 \mathbb{N} \to
\mathbb{Z}$ given by $p(2k) = -2k$ (cf. \cite{Dimca}, Ch. 5). We use
the convention that a sheaf $\sF$ on a variety $Y$ is
perverse if for all integers $k \ge 0$:
\begin{enumerate}
\item $\dim( \supp \sH^m( \sF ) ) < k$ for all $m > -2k$
\item $\dim( \supp \sH^m( \VD\sF ) ) < k$ for all $m > -2k$, where $\VD$ is
  the functor of Verdier duality (cf. \cite{Dimca}, Ch. 3).
\end{enumerate}

Let $d = \dim_{\C} Y$. The conditions above imply that a perverse sheaf on $Y$
has vanishing cohomology sheaves $\sH^m( \sF )$ when $m$ is outside the range $-2d \le m \le 0$ (\cite{Dimca}, Remark 5.1.19).  We denote the constant sheaf
in degree zero with stalk $\C$ at all points of $Y$ by $\C_{Y}$. If $Y$ is
smooth then $\IC{Y} = \C_{Y}[2d]$. 

Suppose that $f : Z \to X$ is a resolution of singularities with the property that for all $i>0$,
\[ \codim \{ x \in X \mid \dim f^{-1}(x) \ge i \} > 2 i .\] 
Such a map is called a \emph{small resolution} since the locus over which fibers can be large is ``small''.

The following proposition connects the topology of $X$ and $Z$ 
with $\IC{X}$. For a proof see \cite{GM2}, Corollary 6.2.

\begin{prop} (Goresky-MacPherson) \label{prop:decomp}
	Let $X$ be an irreducible complex algebraic variety of dimension $s$ and 
	let $f~:~Z~\to~X$ be a small resolution of singularities. Then
	$Rf_* \C_{Z}[2s] \cong \IC{X}$. In particular, for $x \in X$:
	\[ \Chi_x( \IC{X} ) = \sum_i (-1)^{i} \dim H^i( f^{-1}(x); \C ) =
	\Chi( f^{-1}(x) ), \] where $\Chi_x$ denotes the stalk Euler 
	characteristic at the point $x \in X$. 
\end{prop}

Now let $Y$ be a smooth variety equipped with a stratification 
$Y = \cup_{i \in \cS} U_i$.
To each sheaf $\sF$ on $Y$, constructible with respect to the $U_i$, there is an associated cycle,
called the \emph{characteristic cycle}, in the total space of the
cotangent bundle $T^*Y$. The characteristic cycle may be expressed as
$$CC(\sF) = \suml_{i \in \cS} c_i(\sF) \cdot \left[ \, \overline{
      T^*_{U^{\phantom{\circ}}_i} Y} \, \right] .$$
where $c_i(\sF)$ is an integer called the
\emph{microlocal multiplicity} of $\sF$ along $U_i$. 
The cycle may be defined using the
category of holonomic $\mathcal{D}$-modules on $Y$ via the Riemann-Hilbert
correspondence, or explicitly in terms of the complex link spaces of pairs of strata,
see \cite{Dimca}, \S~4.1 .

\subsection{Microlocal Index Formula}

To any constructible sheaf $\sF$ on a smooth stratified variety
$Y=\cup_i U_i$ we described in the previous section two invariants defined for each stratum $U_i$:
$\Chi_i(\sF)$, and $c_i( \sF )$. 
The following theorem, due to  Dubson (\cite{Dubson:index}, 
Th\'{e}or\`{e}me 3) and Kashiwara 
(\cite{Kash}, Theorem 6.3.1), relates these two invariants to the local
Euler obstruction.

\begin{thm}[Microlocal Index Formula]
  \label{thm:microlocal-index}
  For any $j \in \cS$ and
  $\cS$-constructible sheaf $\sF$ on $Y$:
  \[ \Chi_j( \sF ) = \suml_{i \in \cS}
  Eu_{\overline{U_i}}(u_j) \cdot c_i( \sF ) .\]
\end{thm}

The situation we are interested in is when the pair $X \subset Y$ satisfy two conditions. Without loss of generality we can assume that $X$ is the closure of a stratum $U_{i_0}$. First, suppose that $X$ admits a small resolution $f: Z \to X$. Second, suppose that the characteristic cycle of $\IC{X}$ is irreducible, that is:
\[ CC(\IC{X}) = \left[ \, \overline{
      T^*_{U^{\phantom{\circ}}_{i_0}} Y} \, \right] .\]

\begin{prop} \label{prop:dab-eu} Let $U_j \subset X$.
Under the two hypotheses above, the integer $d_{i_0, j}$ from \S\ref{subsec:csm-by-resolution}
equals the local Euler obstruction $Eu_{X}(u_j)$ where $u_j$ is any point in $U_j$.
\end{prop}

\begin{proof}
	Proposition \ref{prop:decomp} implies that
	\begin{equation} \label{eq:dab-chiIC} 
		\Chi_{j}( \IC{X}) = d_{i_0, j}. 
	\end{equation}
  	On the other hand, Theorem \ref{thm:microlocal-index} gives
	\begin{equation}
    	\label{eq:microlocal-grass}
    	\Chi_j( \IC{X} ) =  
    	\sum_i Eu_{\overline{U_i}}(u_j) \cdot c_i( \IC{X} ) .
	\end{equation} 
	Irreducibility of the characteristic cycle implies
	\[ c_i( \IC{X} ) = \left\{
    \begin{array}{ll}
      1 & \mbox{if} \,\, i = i_0 \\
      0 & \mbox{otherwise}
    \end{array} \right. .\] 
  Hence there is a single non-zero term on the right hand side of
  \eqref{eq:microlocal-grass}. 
  Combining \eqref{eq:dab-chiIC} and \eqref{eq:microlocal-grass} we have:
	\begin{equation}
    	\label{eq:dab-eu}
    	d_{i_0, j}  = Eu_{X}( u_j ).
	\end{equation}
\end{proof}

\subsection{Identification of the \tCM Class via Small Resolution} 
\label{subsec:cm-class}

Proposition \ref{prop:dab-eu} leads to the following result
which we derive by comparing formula \eqref{eq:csm-formula-macp}
to formula \eqref{eq:csm-formula-general}. Let $X \subset Y$ and $Z$ satisfy the hypotheses of the last section: $f:Z \to X$ is a small resolution and $\IC{X}$ has irreducible characteristic cycle. We retain the stratification of $Y$ with $X = \overline{U_{i_0}}$.

\begin{thm} 
	\label{thm:cm-xa}
	The \tCM Class of $X$ equals the push-forward of the total Chern
  class of $Z$:
  \[ \CM( X ) = f_* c( TZ ) = f_* \CSM( Z ) .\]
\end{thm}

\begin{proof}
	MacPherson's definition of the \tCSM class as expressed in equation
	\eqref{eq:csm-formula-macp} is
	\begin{equation} \label{eq:macp-eq}
		\CSM( U_{i_0} ) = \suml_{j} f_{i_0, j} \cdot \CM( \overline{U_j} ) .
	\end{equation}
	On the other hand, equation \eqref{eq:csm-formula-general} 
	in our context is
	\begin{equation} \label{eq:my-eq} 
          \CSM( X ) = \suml_{j}
          e_{i_0, j} \cdot (\pi_j)_* \CSM( Z_j ) ,
	\end{equation}
	where $\pi_j : Z_j \to \overline{U_j}$ are compatible resolutions in the sense of \ref{subsec:csm-by-resolution}.
	
	Recall from \S~\ref{subsec:csm-defn} that the integers
        $f_{i, j}$ are entries of the matrix inverse to that
        with entries $Eu_{\overline{U_i}}(u_j)$. Also recall from 
        \S~\ref{subsec:csm-by-resolution}
        that the integers $e_{i, j}$ are entries of the
        matrix inverse to that with entries $d_{i, j}$. Now, proposition
        \ref{prop:dab-eu} implies that $f_{i_0, j} =
        e_{i_0, j}$. It follows by linear algebra that
	\[ \CM(X) = (\pi)_* \CSM(Z) .\]
	Indeed, if we define matrices 
	\begin{enumerate}
		\item $C$ with $(i, j)$ entry the coefficient of
	$[\overline{U_j}]$ in $\CSM(U_i)$
		\item $M$  with $(i, j)$ entry the coefficient of
	$[\overline{U_j}]$ in $\CM(\overline{U_i})$
		\item  $Z$ with $(i, j)$ entry the coefficient of
	$[\overline{U_j}]$ in $(\pi_i)_* \CSM(Z_i)$
	    \item $E = (e_{\alpha, \beta}) = (f_{\alpha, \beta})$
	\end{enumerate}
	then \eqref{eq:macp-eq} is equivalent to
	\[ C = M \cdot E \]
	and \eqref{eq:my-eq} is equivalent to
	\[ C = Z \cdot E .\] The matrix $E$ is invertible since it is
        upper triangular with 1's along the diagonal.  Solving, we have
        $Z = C \cdot E^{-1} = M$ which is equivalent to the statement of
        the proposition.
\end{proof}

\begin{rmk} \label{rmk:stringy}
Suppose that $X$ admits a small crepant resolution and has irreducible characteristic cycle. It is known that the pushforward of the Chern class of a crepant resolution is independant of the resolution and the class obtained is called the \emph{stringy Chern class} of $X$, see \cite{Alu05, deFernex}. In this situation, Theorem \ref{thm:cm-xa} implies that the stringy Chern class and the \tCM class of $X$ coincide.
\end{rmk}

  
\section{Application to Schubert Varieties}
\label{sec:reso}

\subsection{The Grassmannian}
\label{subsec:schub}

We consider the complex Grassmannian
$\GR{k}{n}$, which is the homogeneous space $\GL_n(\C)/P$, where $P$ is
the standard parabolic subgroup consisting of block $k \times k$,
$(n-k) \times (n-k)$ upper triangular matrices. A Schubert cell is an orbit in
$\GR{k}{n}$ of the Borel subgroup $B \subset \GL_n(\C)$ of upper triangular matrices. A
Schubert variety is the closure of a Schubert cell in the Zariski
topology. Schubert cells and varieties in $\GR{k}{n}$ are parametrized
by partitions $\alpha = (\alpha_1 \ge \ldots \ge  \alpha_k)$ such that
$\alpha_1 \le n-k$. We use the notation $\P{k}{n-k}$ to denote the
set of all such partitions. Given a partition $\alpha \in \P{k}{n-k}$,
we denote the corresponding Schubert cell by $\xac$ and its
closure by $\xa$.

Schubert cells partition the Grassmannian and provide a natural
basis for the Chow group (or homology) of $\GR{k}{n}$:
$$ A_* \GR{k}{n} \cong \bigoplus_{\alpha} \mathbb{Z} \cdot [X_\alpha] .$$
In this decomposition, the basis element $[\xa]$ lives in $A_{\abs{\alpha}} \GR{k}{n}$ where $\abs{\alpha} = \sum_i \alpha_i$.
The CSM class of a Schubert variety may be
expressed in this basis as
\begin{equation}
\label{eq:a-priori-csm}
\csm(\xa) = [X_\alpha] + \sum_{\beta < \alpha} \gamma_{\alpha,
  \beta} [X_\beta] .
\end{equation} 
In these terms, the main goal of the remainder of the paper is to understand the
coefficients $\gamma_{\alpha,\beta}$.

In the Grassmannian, Equation \eqref{eq:csm-formula-general} from \S\ref{subsec:csm-by-resolution} is interpreted as
follows. The stratum $S_i$ is a Schubert cell $\xac$. Zelevinsky's
construction, which is recalled in \S~\ref{sec:reso}, provides the
family of resolutions $\pa : \za \to \xa$ and the integers $d_{i,j}$
become $\dab$, the Euler characteristic of the fiber of $Z_\alpha$
over a point in $\xbc$. These integers are calculated using
Zelevinsky's formula which is stated in Theorem
\ref{thm:zel-formula}. Similarly, the matrix inverse to the matrix of
integers $\dab$ has entries denotes $\eab$. With this notation,
\eqref{eq:csm-formula-general} becomes
\begin{equation} \label{eq:csm-formula-grassmannian}
  \CSM(\xac) = \sum_{\beta \le \alpha} \eab \cdot \pas \CSM(\zb) .
\end{equation}
In Section \ref{subsec:csm-reso} we describe $T\zb$ as a vector bundle in terms of pullbacks of
universal bundles and use this to compute
$\csm(\zb) = c(T\za) \cap [\zb]$. 

\subsection{Resolution of Singularities}
\label{sec:reso}
\label{subsec:zel-reso}

We recall Zelevinsky's resolution of singularities for
a Schubert variety in the Grassmannian. In \S~\ref{subsec:csm-reso} we
describe the tangent bundle of a Zelevinsky resolution inductively and
give two formulas for its total Chern class in terms of pullbacks of
universal bundles from Grassmannians. Since the resolution is smooth,
this immediately gives a formula for its CSM class. The reader may refer to \cite{Zel} and \cite{BFL} for details on the construction and proofs of the main results below.

Let $V \cong \C^n$ and $V_\bullet = (V_0 \subset V_1 \subset \cdots
\subset V_n)$ be a fixed complete flag in $V$.  Let $X_\alpha \subset
\GRb{k}{V}$ be a Schubert variety relative to $V_\bullet$ and indexed
by the partition $\alpha$. 

We use peak notation for a partition $\alpha = (\alpha_1 \ge \cdots
\ge \alpha_k)$.
Let $a_i$ be the multiplicity of the $i$-th distinct part of $\alpha$
ordered from smallest to largest. The $b_i$ are the unique positive
integers such that the $i$-th distinct part of $\alpha$ is $b_0 +
\cdots + b_{i-1}$. In peak notation we write $\alpha = [a_1, \ldots,
  a_m ; b_0, \ldots, b_{m-1}]$. For $0 \le i \le m$, let $V^i = V_{a_1
  + \cdots + a_i + b_0 + \cdots + b_{i-1}}$. The Schubert variety
$X_\alpha$ is then defined by intersection conditions:
$$X_\alpha = \left\{ U \in \GR{k}{n} \mid \dim U \cap V^i \ge a_1 +
\cdots + a_i \quad \text{for} \quad 0 \le i \le m \right\} .$$
It is common to visualize $\alpha$ in peak notation by a piece-wise
linear function $f_\alpha$ defined on the real line, see Figure
\ref{fig:peak-diagram}. The local maxima of $f_\alpha$ are called
\emph{peaks} (labeled $P_i$) and the local minima are called
\emph{depressions} (labeled $Q_i$), see \cite{Jones:Thesis}, Ch. 5 for a
thourough discussion with examples. 

\begin{figure}[ht]
\includegraphics{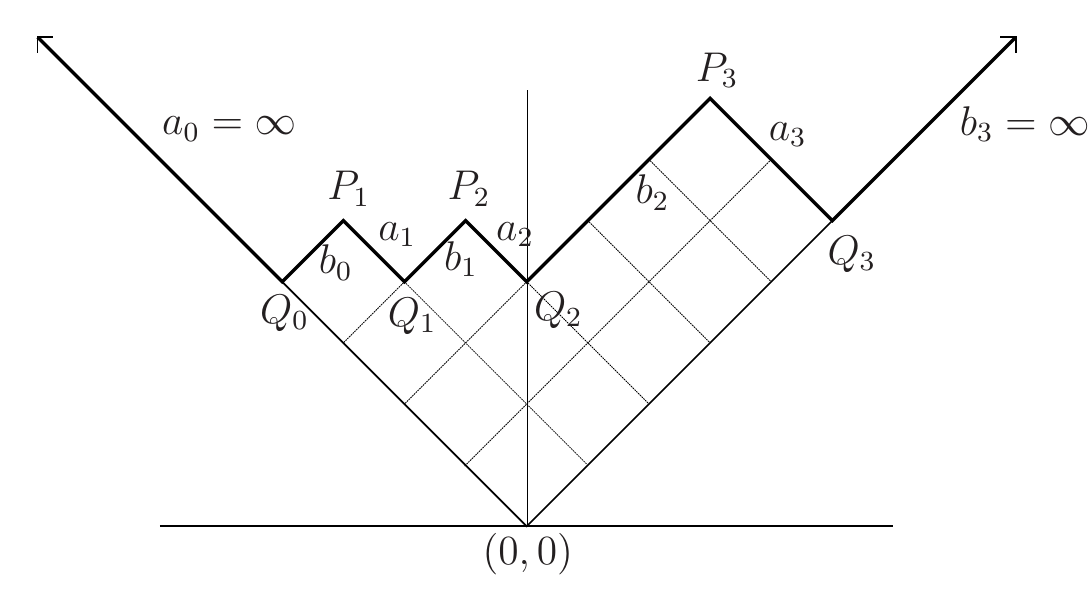}
\caption{Peak diagram for $\alpha = [1,1,2; 1,1,3] = (5,5,2,1)$.}
\label{fig:peak-diagram}
\end{figure}

Let $s$ be an ordering of the peaks of $\alpha$, i.e. $s$ is a permutation of
$\{1, \ldots, m\}$. For each choice of permutation, Zelevinsky has
constructed a variety over $X_\alpha$:
$$ \zap \overset{\pap}{\longrightarrow} X_\alpha
.$$
$\zap$ is a subset of a certain product of
Grassmannians:
$$ \zap \subset \GRb{k_1}{V} \times \GRb{k_2}{V} \times \cdots \times
\GRb{k_m}{V} ,$$ where $k_m = k$. The map $\pap$ is the restriction to
$\zap$ of projection to the last factor. For each $i \in \{1, \ldots,
m\}$, two subspaces are defined which constrain the points in the
$\GRb{k_i}{V}$ factor. 
\begin{equation} \label{eq:zel-defn}
 \zap = \left\{ U_\bullet \in \prod_i \GRb{k_i}{V} \mid \lW_i \subset U_i \subset \rW_i
\right\} .
\end{equation}

The definition of $\lW_i$ and $\rW_i$ is inductive. Let $j_1 = s(1)$,
$k_1 = \dim V^{j_1-1} + a_{j_1}$, and $U_1$ denote a point in the
Grassmannian $\GRb{k_1}{V}$ with the property that
$$V^{j_1-1} \subset U_1 \subset V^{j_1} .$$ 
Set $\lW_1 = V^{j_1-1}$ and $\rW_1 = V^{j_1}$. For such a $U_1$,
consider the new partial flag $^1V^\bullet = ( V^0 \subset \cdots
\subset V^{j_1-2} \subset U_1 \subset V^{j_1+1} \subset \cdots \subset
V^m)$. For $i \ge 2$, $\GRb{k_i}{V}$, $\lW_i$, and $\rW_i$
are defined inductively using the flag  $^1V^\bullet$, a new partition $\alpha'$ 
which has $m-1$ peaks, and the permutation $s'$ of $\{ 1, \ldots, m-1
\}$ induced by $s$ and the order preserving bijection $\{ 1, \ldots, \hat{j_1},
\ldots, m \} \iso \{ 1, \ldots, m-1 \}$. The partition $\alpha'$ is obtained by removing peak
$P_{j_1}$ from $\alpha$ in the following way. We replace the
pair $a_{j-1}, a_{j}$ (in peak notation) by their sum $a_{j-1} + a_{j}$ and similarly replace
$b_{j-1},b_{j}$ by $b_{j-1} + b_{j}$. For the extreme cases, $j=1$ or
$m$, we use the convention that $a_0 = b_m = \infty$ and drop these
from the notation. For example, removing the 1st peak from $[1,1,2; 1,
1, 3]$ results in the partition $[1,2; 2,3]$, while removing the 2nd
peak results in the partition $[2, 2; 1, 4]$.

The result of Zelevinsky's construction is a closed subvariety of
$\prod_i \GRb{k_i}{V}$ defined by incidence relations $\lW_i \subset
U_i \subset \rW_i$. By construction, $\lW_i$ and $\rW_i$ are either
points $U_j \in \GRb{k_j}{V}$ for $j < i$ or they are subspaces in the
fixed partial flag $V^\bullet$ determined by $\alpha$.

\begin{ex} \label{ex:zel-reso-21}
  Let $V = \C^4$ and the fixed flag be $V_\bullet = (0 \subset \C^1 \subset \C^2
  \subset \cdots \subset \C^4)$. For $\alpha = (2,1)$, the partial
  sub-flag is $V^\bullet = (0 \subset \C^2 \subset \C^4)$. The results of
  Zelevinsky's construction when $s = \id$ (the identity permutation)
  and $s = w_0$ (the order reversing permutation) are:
$$
\begin{array}{l}
Z_{(2,1), \id} = \left\{ (U_1, U_2) \in \GR{1}{4} \times \GR{2}{4}
  \mid 0 \subset U_1 \subset \C^2, \,\, U_1 \subset U_2 \subset \C^4
\right\} \\
\\
Z_{(2,1), w_0} = \left\{ (U_1, U_2) \in \GR{3}{4} \times \GR{2}{4}
  \mid \C^2 \subset U_1 \subset \C^4, \,\, 0 \subset U_2 \subset U_1
\right\} .
\end{array}
$$
\end{ex}

The group $\GL(V)$ acts on each $\GRb{k_i}{V}$ and diagonally on the
product $\prod_i \GRb{k_i}{V}$. It is clear from description above
that the Borel subgroup $B$  which fixes $V_\bullet$ leaves
$\zap$ stable. Also, the map $\pap$ is clearly $B$-equivariant.

We now summarize the main theorems of \cite{Zel}.

\begin{thm} \label{thm:zel-thm} (Zelevinsky)
\begin{enumerate}
\item For any $\alpha$ and  $s$, $\zap$ is a smooth projective variety.
\item For a fixed $\alpha$ and any $s$, $\pap : \zap \to \GR{k}{n}$ is a proper map with
  image $\xa$ and is an isomorphism over $\xac$, in particular it
  is a resolution of singularities.
\item For all $\alpha$ there exists a permutation $s$ such that 
$\pap : \zap \to \xa$ is a small resolution.
\end{enumerate}
\end{thm}

For $\alpha, s$ as above, consider a partition $\beta \le \alpha$. 
Let $U \in \xbc \subset \GRb{k}{V}$. The intersection conditions that
a point $W \in \xa$ satisfies with respect to $V^\bullet$ are $\dim (W \cap
V^i) \ge a_1 + \cdots + a_i$ for all $i=0, \ldots, m$. In particular,
$U$ satisfies these conditions.
Define a sequence of integers $(c_0, c_1, \ldots, c_m)$ , 
called the relative depth vector for the pair
$(\alpha, \beta)$, by $c_i = \dim ( U \cap V^i ) - (a_1 + \ldots +
a_i)$. The following formula calculates
$d_{\alpha,s}(\beta) := \Chi( \pap^{-1}(U) )$ in terms of $\alpha$, $s$, and
the vector $(c_0, \ldots, c_m)$.

\begin{thm} \label{thm:zel-formula} (Zelevinsky) Let $\alpha = [a_i ;
  b_i]$ and let $j = s(1)$, the
  index of the first peak to be removed from $\alpha$ in the
  construction of $\zap$. Let $\alpha'$ be the partition which results
  from this peak removal and $s'$ be the peak ordering for $\alpha'$
  induced by $s$. Then,
$$ d_{\alpha,s}(\beta)  = 
\suml_{d \ge 0} \binom{a_j+c_j-c_{j-1}}{c_j - d}
\binom{b_{j-1} - c_j + c_{j-1}}{c_{j-1}-d} \cdot d_{\alpha', s'}(\gamma(d)) $$
where $\gamma(d)$ is any partition such that $\gamma(d) \le \alpha'$ and
the pair $(\alpha', \gamma(d))$ has relative depth vector $(c_0, \ldots,
c_{j-2}, d, c_{j+1}, \ldots, c_m)$. 
\end{thm}

\begin{rmk}
From the observations made in Proposition \ref{prop:dab-eu} and \S~\ref{subsec:cm-computation} it follows that  if $\pap$ is a small
resolution, then 
$d_{\alpha,s}(\beta)$, regarded as a constructible function on $\xa$,
coincides with the local Euler obstruction. 
\end{rmk}

\begin{prop} \label{prop:zel-chi-fiber}
Let $\pap : \zap \to \xa$ be a small resolution. Then $U \in \xbc
\subset \xa$ is a smooth point of $\xa$ if and only if any of the following
equivalent conditions hold:
\begin{enumerate}
\item $\pap^{-1}(U)$ is a point
\item $d_{\alpha, s}(\beta) = 1$
\item The relative depth vector for the pair $(\alpha, \beta)$ is
  $(0,0,\ldots,0)$.
\end{enumerate}
\end{prop}

\begin{proof}
  First, $U$ is a smooth point if and only if $U$ is a rationally
  smooth point by the equivalence of smoothness and rational
  smoothness for Schubert varieties in simply laced groups (in
  particular for $\GL_n(\C)$), see \cite{Carrell-Kuttler:SmoothPts}
  (ADE-Theorem) and \cite{Deodhar:LocalPoincare}. Rational smoothness
  in turn is equivalent to the statement that $\Chi_U(\IC{\xa})$, the
  stalk Euler characteristic of the intersection cohomology sheaf of
  $\xa$ at $U$, equals 1. Indeed, the odd dimensional intersection
  cohomology vanishes, so $\Chi_U(\IC{\xa}) = 1$ is equivalent to
  having 1 dimensional cohomology only in the smallest possible degree
  (see the conventions in \S~\ref{sec:micro-local}) which means that
  the point is rationally smooth by definition.

  Since the resolution is small, $\Chi_U(\IC{\xa}) = \Chi(
  \pap^{-1}(U) )$, see Proposition \ref{prop:decomp}. Fibers over points in a
  Zelevinsky resolution have vanishing odd dimensional cohomology (see
  \cite{Zel}, Theorem 2), hence $\Chi( \pap^{-1}(U) ) = 1$ if and only
  if $\pap^{-1}(U) = \pt$. We have thus established that the statement is
  equivalent to both (1) and (2). 

  To see that (2) and (3) are equivalent, we argue using the formula
  of Theorem \ref{thm:zel-formula}. First, suppose that
  $d_{\alpha,s}(\beta) = 1$. Then there can be only a single
  term in the formula. It follows that the only term appearing is the
  $d=0$ term (if a $d>0$ term is non-zero, then terms for smaller $d$
  will also be non-zero). It also follows that either $c_j =0 $ or
  $c_{j-1} = 0$ since if they were both non-zero there would be at
  least a $d=0$ and $d=1$ term. 

  We are left with the statement that $1 =
  \binom{a_j+c_j-c_{j-1}}{c_j} \cdot
  \binom{b_{j-1}-c_j+c_{j-1}}{c_{j-1}} d_{\alpha', s'}(\gamma(0))$.
  So each factor on the right hand side equals $1$. By induction, we may assume that
  $d_{\alpha',s'}(\gamma(0)) = 1$ implies $$(c_0, c_1, \ldots, c_{j-2},
  d, c_{j+1}, \ldots, c_m) = (0, \ldots, 0) .$$ It remains to show that
  $c_{j-1} = c_j = 0$. Since $\binom{a_j+c_j-c_{j-1}}{c_j} = 1$ we
  have that either $a_j-c_{j-1}=0$ or $c_j = 0$. Similarly we have
  that either $b_{j-1}-c_j=0$ or $c_{j-1}=0$. But now since we know
  that one of the two integers $c_j, c_{j-1}$ is zero, it follows that
  the other is zero because $a_j, b_{j-1} > 0$. Thus we have
  established that (2) implies (3). The converse is a simple exercise
  and left to the reader.
\end{proof}

It is convenient to encode the incidence relations which define $\zap$
using a directed graph. Each relation is of the form $\lW_i \subset
U_i \subset \rW_i$ where $\lW_i, \rW_i$ denote either arbitrary points
in one of the factors $\GRb{k_j}{V}$ or fixed subspaces from the
partial flag $V^\bullet$. Consider the directed graph whose vertices
are the $U_i$ along with the $V^i$ and which has a directed edge
corresponding to each inclusion relation which holds among points in
the factors of $\zap$.

For example the directed graph 
		\[ 
		\xygraph{
		!{<0cm,0cm>;<1cm,0cm>:<0cm,-1cm>::} 
		!{(0,0) }*+{\bullet_{V^0}}="c0" 
		!{(2,0) }*+{\bullet_{V^1}}="c1" 
		!{(4,0) }*+{\bullet_{V^2}}="c2" 
		!{(3,1)}*+{\bullet_{U_1}}="n1"
		!{(2,2)}*+{\bullet_{U_2}}="n2"
		"c0":"c1" "c1":"c2"
		"c1":"n1" "n1":"c2"
		"c0":"n2" "n2":"n1"
		} .\]
represents the incidence relations occurring in the definition of
$\zap$ where $\alpha$ has two peaks and $s = w_0$ is the order
reversing permutation in $S_2$. See Example \ref{ex:zel-reso-21} for
comparison. If we also attach a non-negative integer to each vertex
giving the dimension of the abstract subspace it represents then this
weighted, directed graph completely determines $\zap$.

\subsection{CSM Class of a Resolution}
\label{subsec:csm-reso}

We give a formula for the total Chern class of the tangent bundle of a
Zelevinsky resolution in terms of the pullbacks of universal
bundles. Since $\zap$ is smooth, this gives a formula for the CSM
class $\csm(\zap)$. We retain the assumptions and notation from
\S~\ref{subsec:zel-reso}. For the remainder of the section we fix
$\alpha$ and $s$ and let $Z = \zap$ and $X = \Grass$.

Let $\pr_i: X \to \GRb{k_i}{\V}$ be the  projection map.
Let $\bV^i$ be the trivial rank $i$ bundle on $X$ whose fiber over
any point is identified with $\V^i$ from $\V^\bullet$. Let
$\bU_i$ be the universal sub-bundle of $\bV$ on $\GRb{k_i}{\V}$
as well as its pull-back under $\pr_i$ to $X$ and $Z$ by abuse
of notation.

For each $i \in \{1, \ldots, m\}$, we define two vector bundles
$\blW_i, \brW_i$ on $X$ by letting the fiber of $\blW_i, \brW_i$
over a point $U_\bullet$ be $\lW_i$, $\rW_i$
respectively. Thus each of $\blW_i$ and $\brW_i$ is either a universal
sub-bundle $\bU_j$, or a trivial bundle $\bV^j$.  Then for each
incidence relation $\lW_i \subset U_i \subset \rW_i$ defining the
variety $Z$, there corresponds an incidence relation among vector
bundles on $Z$: $\blW_i \subset \bU_i \subset \brW_i$.

Let $X^{(j)} = \prod_{i=1}^{j} \GRb{k_i}{V}$ and let $\rho_j : X \to
X^{(j)}$ denote the projection away from the last $m-j$ factors. Thus
$\rho_j(U_1, \ldots, U_m) = (U_1, \ldots, U_{j})$. Similarly, let
$Z^{(j)} = \rho_j(Z)$.  There are
natural projection maps $Z^{(j)} \to Z^{(l)}$ whenever $j >
l$. Finally, let $l_i := k_i - \dim \lW_i$.

\begin{prop} \label{prop:Z-fiber-bundle}
For $j \in \{2, \ldots, m\}$, the natural map $Z^{(j)} \to
Z^{(j-1)}$ is a Grassmannian bundle whose fiber may be identified
with $\GRb{l_{j}}{ \rW_j / \lW_j }$. Moreover, $Z^{(1)}$ 
is isomorphic to the Grassmannian $\GRb{l_1}{ \rW_1 / \lW_1 }$.
\end{prop}

\begin{proof} The variety $Z^{(j)}$ is identified with
$$Z^{(j)} = \left\{ (U_1, \ldots, U_{j}) \mid \lW_i \subset U_i
  \subset \rW_i \right\} $$
and similarly for $Z^{(j-1)}$. If we fix a point $(U_1, \ldots,
U_{j-1}) \in Z^{(j-1)}$ then the fiber may be identified with the
set
$$\left\{ U_{j} \in \GRb{k_{j}}{V} \mid \lW_{j} \subset
  U_{j} \subset \rW_{j} \right\} .$$ Zelevinsky's construction
guarantees that $\lW_{j} \subsetneq \rW_{j}$, hence the fibers are all
non-empty and naturally isomorphic to a Grassmannian. It follows that
$Z^{(j)}$ is naturally the Grassmannian bundle over $Z^{(j-1)}$
of $l_{j}$ dimensional subspaces of the fibers of the
vector bundle $\brW_{j} / \blW_{j}$.

The second part of the proposition follows since $\lW_1$ and
$\rW_1$ are both elements of the fixed partial flag $V^\bullet$.
\end{proof}

Let $Y$ be a smooth variety equipped with a vector bundle $\bE \to
Y$. Then $\GRb{k}{\bE}$ denotes the Grassmannian bundle of
$k$-dimensional subspaces of the fibers of $\bE$ over $Y$ and $\pi :
\GRb{k}{\bE} \to Y$ is the map to $Y$. $\GRb{k}{\bE}$
has three natural vector bundles over it. 
There is the pullback $\pi^{-1} \bE$, the
universal sub-bundle $\bS$ whose fiber over a point
$(y, V)$ is the subspace $V$ of $\bE(y)$, and there is the universal 
quotient bundle $\bQ$ whose fiber over $(y, V)$ is $\bE(y) / V$. Thus
there is an exact sequence of vector bundles on $\GRb{k}{\bE}$:
\[ 0 \to \bS \to \pi^{-1} \bE \to \bQ \to 0 .\]

The map $\pi : \GRb{k}{\bE}
\to Y$ is a fiber bundle with smooth fibers. Consequently, $\GRb{k}{\bE}$
is non-singular and its tangent bundle sits in an exact sequence
\[ 0 \to \pi^{-1} TY \to T\GRb{k}{\bE} \to T_{\GRb{k}{\bE}/Y} \to 0 .\]

\begin{prop} \label{prop:vertical-bundle} (\cite{Ful}, B.5.8) The
  relative tangent bundle $T_{\GRb{k}{\bE}/Y}$ is canonically
  isomorphic to
	\[ \Hom(\bS, \bQ) \cong  \bS^\vee \otimes \bQ .\]
\end{prop}

We now describe the tangent bundle of $Z$ in terms of universal
bundles on $X$.

\begin{thm}
\label{thm:c-tangent-bundle}
  \[ c(TZ) = \prodl_{i=1}^m c( (\bU_i/\blW_i)^\vee \otimes
  (\brW_i/\bU_i) ) \]
\end{thm}

\begin{proof}
We proceed by showing that
 \[ c(TZ^{(j)}) = \prodl_{i=1}^{j} c( (\bU_i/\blW_i)^\vee \otimes
  (\brW_i/\bU_i) ) \]
by induction on $j$.

If $j = 1$, then $Z^{(j)}$ is a sub-Grassmannian of $\GRb{k_1}{V}$ defined
by a single incidence relation $\lW_1 \subset U_1 \subset \rW_1$.
The universal sub-bundle on $Z^{(1)}$ is $\bU_1/\blW_1$ and the quotient bundle
is $\brW_1/\bU_1$. Proposition \ref{prop:vertical-bundle}  implies that
we have a canonical isomorphism of vector bundles:
\[ TZ^{(1)} \cong (\bU_1/\blW_1)^\vee \otimes (\brW_1/\bU_1) .\] Taking total
Chern classes of both sides the proposition follows.

Suppose that $1 < j \le m$.  The previous proposition says that $Z^{(j)}$ is a
Grassmannian bundle over $Z^{(j-1)}$, specifically
\[ Z^{(j)} \cong \GRb{l_j}{\brW_j/\blW_j} \]
and the bundle map $\GRb{l_j}{\brW_j/\blW_j} \to Z^{(j-1)}$ is the
restriction of the projection map $\phi: X^{(j)} \to X^{(j-1)}$.
The universal sub-bundle on $\GRb{l_j}{\brW_j/\blW_j}$ is 
$\bU_j/\blW_j$ and the quotient bundle is $\brW_j/\bU_j$.
We have an exact sequence of vector bundles on $Z^{(j)}$,
\[ 0 \to \phi^{-1} TZ^{(j-1)} \to TZ^{(j)} \to T_{Z^{(j)}/Z^{(j-1)}} \to 0 ,\]
and by Proposition \ref{prop:vertical-bundle} there is a canonical
isomorphism of vector bundles
\[ T_{Z^{(j)}/Z^{(j-1)}} \cong (\bU_j/\blW_j)^\vee \otimes (\brW_j/\bU_j) .\]
The standard properties of Chern classes (\cite{Ful}, Theorem 3.2)
and the inductive hypothesis imply that
\[ c(TZ^{(j)}) = \prodl_{i=1}^{j} c( (\bU_i/\blW_i)^\vee \otimes
  (\brW_i/\bU_i) ) .\]
\end{proof}

For explicit calculations it is convenient to reformulate Theorem
\ref{thm:c-tangent-bundle} in terms of Schur determinants in Segre classes
of the bundles $\bU_i$, $\blW_i$, and $\brW_i$. 

\begin{defn} \label{defn:D-lambda-mu} Given two partitions $\lambda,
  \mu$ with exactly $N$ parts (the parts are allowed to be zero here), define
\[ D^N_{\lambda, \mu} = \det \left( \binom{\lambda_i + N - i}{\mu_j + N - j} 
\right)_{1 \le i,j \le N} \] 
where $\binom{\lambda_i + N - i}{\mu_j + N - j} = 0$ by convention if
$\lambda_i + N - i < \mu_j + N - j$.
\end{defn}

\begin{rmk} \label{rmk:D-lambda-mu-pos}
	The integer $D^N_{\lambda, \mu}$ defined above is always non-negative 
	(cf. \cite{MacD}, p. 30 - 31). It may be interpreted as the number in a certain
	family of non-intersecting lattice paths via the Lindstrom-Gessel-Viennot Theorem 
	(cf. \cite{Krat}, Theorem 6). Also, $D^N_{\lambda,\mu}$ depends on $N$ in the
        sense that for $\lambda, \mu$ with exactly $N$ parts, $D^N_{\lambda,
          \mu}$ is generally not equal to $D^{N+1}_{(\lambda,0), (\mu, 0)}$.
\end{rmk}

For positive integers $a,b$
let $a \times b$ denote the partition $(b^a) = (b, \ldots, b)$.
For a partition $\lambda \in \P{k}{N}$, let
$\widetilde{\mathbb{D}}(\lambda)$ be the conjugate partition to
$\Dual_{k \times N}(\lambda)$, where the latter notation indicates the
dual of $\lambda$ in  $k \times N$: $$\Dual_{k \times N}(\lambda) = (N-\lambda_k,
N-\lambda_{k-1}, \ldots, N-\lambda_1) .$$ In the formula below,
$s_\lambda(\bB)$ is the Schur determinant corresponding to $\lambda$
evaluated at the Segre classes of $\bB$. See \cite{Ful}, \S~14.5 and
note that we are following the sign convention of \cite{FP} regarding
Segre classes.

\begin{lem} \label{lem:schur-ctz} (\cite{Ful}, Example 14.5.2)
	Let $\bE, \bF$ be vector bundles over a variety $X$ 
	of rank $e, f$ respectively.  Then
	\[ c(\bE^\vee \otimes \bF) = 
	\suml_{\mu \subset \lambda \subset e \times f} D^e_{\lambda, \mu} \,
	s_\mu(\bE^\vee) \, s_{\widetilde{\mathbb{D}}(\lambda)}(\bF)
	 .\]
\end{lem}

Let $l_i := \rk \bU_i - \rk \blW_i$ and $r_i := \rk \brW_i - \rk \bU_i$.
The lemma and Theorem \ref{thm:c-tangent-bundle} combine to give
the following corollary.

\begin{cor} \label{cor:alt-c-tangent-bundle}
	\[ c(TZ) = \prodl_{i=1}^p \left( 
	\suml_{\mu \subset \lambda \subset l_i \times r_i} D^{l_i}_{\lambda, \mu}  \,
	s_\mu((\bU_i/\blW_i)^\vee) \, 
	s_{\widetilde{\mathbb{D}}(\lambda)}(\brW_i/\bU_i)
	\right)
	 \]
\end{cor}

\begin{rmk} \label{rmk:uses-of-resolutions} In Section
  \ref{sec:csm-reso} we focus on using the formulas above to calculate
  CSM classes of Schubert cells and varieties. Following the method
  outlined in Section \ref{subsec:csm-by-resolution}, one may choose
  the family of compatible resolutions arbitrarily. This means that
  for each parition $\alpha$, the choice of permutation $s$ (and hence
  Zelevinsky resolution $\zap$) is arbitrary. However, some choices
  are more suitable for particular applications.  To apply the results of
   Section \ref{sec:micro-local}, resolutions are required to be
  \emph{small} and
  thus the permutation must be chosen accordingly (see Theorem
  \ref{thm:zel-thm}, part 3). For performing calculations by computer,
  resolutions corresponding to the identity or the order reversing
  permutation are preferable for their combinatorial simplicity. The
  resolutions corresponding to order reversing permutations are
  used in Section \ref{subsec:direct-pushforward} to give another
  method of calculating the key pushforward $\pas$ appearing in
  \eqref{eq:csm-formula-grassmannian}.
\end{rmk}


\section{Pushforward to the Grassmannian}
\label{sec:csm-reso}

In this section we compute the pushforward of the CSM class of a
resolution to the Grassmannian. This gives one ingredient of formula
\eqref{eq:csm-formula-grassmannian}.

An algorithm for calculating the pushforward of $\CSM(Z)$ is described
for an arbitrary Zelevinsky resolution in Section
\ref{subsec:integration} by means of the Bott Residue Formula. A more
direct method for the case when $Z = \zap$ with $s$ the order
reversing permutation is described in
Section \ref{subsec:direct-pushforward}. Here, the main tool is a
description of the Gysin map of a Grassmannian bundle due to
Fulton-Pragacz. 

\subsection{Integration} \label{subsec:integration}

The pushforward map $\pi_*$, associated to a proper map of varieties
$$\pi : X \to Y ,$$ may be understood via integration (in the algebraic
sense) over $X$. When $X$ admits an action of an algebraic torus, $T$,
the integration may be carried out explicitly using localization for
the $T$-equivariant Chow group of $X$. 

Let $X$ be a smooth projective $T$-variety. Let $\bE_1, \ldots, \bE_n$
be $T$-equivariant vector bundles on $X$ and let $P(\bE)$ denote a
polynomial in the Chern classes of the $\bE_i$. There are equivariant
Chern classes $c_j^T(\bE_i)$ and a corresponding polynomial $P^T(\bE)$
in the equivariant Chern classes which specializes to $P(\bE)$. We use
the notation $[Y]_T$ to denote the $T$-equivariant fundamental class
of a $T$-variety $Y$. See \cite{EG:IntersectionTheory,
  EG:Localization} for details.

It is well known that the fixed point set $X^T$ is smooth with
finitely many connected components. If $F$ is a connected component,
let $N_F X$ be the normal bundle to $F$ in $X$, and $d_F = \rk N_F X =
\codim_X F$. The following theorem is an algebraic form of the Bott
Residue formula which may be found in \cite{EG:Localization}, Theorem
3.

\begin{thm} (Bott residue formula)
$$\int_X P(\bE) \cap [X] = \sum_{F \subset X^T} \int_F \frac{
  P^T(\bE\lvert_F) \cap [F]_T}{c^T_{d_F}(N_F X)} .$$
\end{thm}

We now return to the context of the Grassmannian. Let $T \subset B$ be a
maximal torus and Borel subgroup of $G = \GL_n(\C)$ which stabilize
the fixed flag $V_\bullet \subset \C^n$ from \S \ref{subsec:zel-reso} which we used to
define Schubert varieties. Let $Z = \zap$ as in
\S~\ref{subsec:csm-reso}. The ambient space of the resolution $Z$ is
$X = \prod_i \GR{k_i}{n}$. $G$ acts diagonally on $X$ and $Z$ is $B$
invariant for the action (see \S~\ref{subsec:zel-reso}).  Hence $Z$ is
also $T$ invariant.

To calculate coefficients in the Schubert basis of the class $\pi_*
\csm(Z)$ it suffices to use the Bott Residue Formula to calculate 
integrals as follows.

\begin{lem} \label{lem:how-to-integrate}
Let $\gamma_{\alpha, \beta}$ be the coefficient of $[\xb]$ in $\pi_*
\csm(Z)$. Then,
$$ \gamma_{\alpha, \beta} = \int_Z c(TZ) \cdot
c_{\beta}(\bC^n/\bU) \cap [Z]$$
where $\bC^n/\bU$ is the pullback of the universal quotient bundle on
$\GR{k}{n}$. 
\end{lem}

As in \S \ref{subsec:csm-reso}, $c_\beta(-)$ denotes a
Schur determinant in the Chern classes of the argument.  The lemma
follows from a routine application of the projection formula and the
duality formula for Schubert calculus (see \cite{Ful}, \S14.7). The
integrand in the Lemma consists of Chern classes of $T$-equivariant
bundles on $Z$, thus the Bott Residue Formula applies.

\begin{lem} \label{lem:finite-fixed-pts}
The fixed point set $Z^T$ is finite.
\end{lem}
\begin{proof}
This is clear since $Z \subset \prod_i \GR{k_i}{n}$ and each factor 
$\GR{k_i}{n}$ has finitely many $T$ fixed points: each Schubert cell contains a unique point, its `center', which is fixed by $T$. 
\end{proof}

To make the integral in Lemma \ref{lem:how-to-integrate} explicit, one
needs to parametrize the fixed point set and express the equivariant
Chern classes restricted to the fixed points in terms of characters of
$T$ (see \cite{EG:Localization}, Lemma 3).

Let $\{ e_i \}_{i=1 \ldots n}$ be a basis of $\C^n$ so that $T$ is the
group of diagonal matrices in $G$ and the flag fixed by $B$ is $0
\subset \Span{e_1} \subset \Span{e_1,e_2} \subset \cdots \subset
\Span{e_1, \ldots, e_n}$. Then the $T$ fixed points in $\GR{r}{n}$ are
the subspaces $\Span{e_{i_1}, \ldots, e_{i_r}}$ where $(i_1, \ldots,
i_r)$ is any sequence of integers such that $1 \le i_1 < \cdots < i_r
\le n$. Therefore, a point $U_\bullet \in \prod_i \GR{k_i}{n}$ is in
$Z^T$ if and only if
\begin{enumerate}
\item For each $i$, $U_i = \Span{e_{j^i_1}, \ldots, e_{j^i_{k_i}}}$
for some sequence $j^i_1 < \cdots < j^i_{k_i}$
\item The subspaces $U_i$ satisfy all the incidence relations defining
  $Z$. 

\end{enumerate}

Let $U_\bullet \in Z^T$. The restriction to $U_\bullet$ of a
$T$-equivariant bundle (such as the pullback of a universal sub or
quotient bundle to $Z$) is a representation for $T$ which decomposes
according to the characters. The equivariant Chern classes of such a
restriction may be calculated in terms of the characters (see
\cite{EG:Localization} Lemma 3). For our purposes, the following
special case suffices.  Let $\hat{T}$ be the character group of $T$
and let $R_T \cong \Sym(\hat{T}) \cong \mathbb{Z}[t_1, \ldots, t_n]$
denote the $T$-equivariant Chow ring of a point.

\begin{lem} \label{lem:equiv-Chern-via-character}
Let $E_\chi \to \pt$ be a rank $r$, $T$-equivariant vector bundle over a point
such that the action of $T$ on any vector in $E_\chi$ is given by
the character $\chi$. Then,
$$ c_i^T(E_\chi) = \binom{r}{i} \chi^i \in R_T .$$
\end{lem}

To carry out explicit calculations it is useful to restrict the torus
action to a one-parameter subgroup $\C^* \subset T$ such that $Z^{\C^*} =
Z^T$. This can be achieved by letting the $\C^*$ act with generic
enough weights. We identify the $\C^*$-equivariant Chow ring of a
point with $\mathbb{Z}[t]$.

\begin{ex} \label{ex:integration}
We illustrate the calculation of $\gamma_{(2,1), (1,1)}$ using the
Bott Residue Formula. The coefficient depends not only on the
partitions $(2,1)$ and $(1,1)$, but also a choice of resolution for
$X_{(2,1)}$. Let $\pi_{(2,1)} : Z \to X_{(2,1)}$ denote the
Zelevinsky resolution corresponding to the identity permutation. This
is a small resolution. It is an isomorphism away from the point
$X_{(0)} \subset X_{(2,1)}$ and has fiber isomorphic to $\CP^1$ over
$X_{(0)}$.

Evaluating $c(TZ)$ according to Corollary
\ref{cor:alt-c-tangent-bundle} and selecting the terms with
appropriate degree we have:
\begin{equation} \label{eq:gamma-integral} 
\gamma_{(2,1),(1,1)} = \int_Z \left(
3 c_1(\bU_1^\vee) + 3 c_1( (\bU_2/\bU_1)^\vee )
\right) \cdot c_{(1,1)}(\bV/\bU_2) \cap [Z] .
\end{equation}

Using the Bott Residue Formula, we calculate the integral
\begin{equation} \label{eq:bott-res-integral}
\int_Z c_1(\bB) c_{(1,1)}(\bV/\bU_2) \cap [Z] .
\end{equation}
Below, $\bB$ will be either $\bU_1^\vee$ or
$(\bU_2/\bU_1)^\vee$.  Suppose that $\C^* \subset T$ acts on $V =
\C^4$ with weights $[w_1,\ldots,w_4] = [0,1,2,3]$. These weights are
generic enough so that $Z^{\C^*} = Z^T$.

To describe the $T$ fixed points in $Z$, let $i_1 \in \{1,2\}$ and
$i_2 \in \{1,2,3,4\} \backslash \{i_1\}$. Then a $T$ fixed point has
the form $U_\bullet = ( \Span{e_{i_1}} , \Span{e_{i_1},e_{i_2}} )$ for some
choice of $i_1, i_2$. In this case there are 6 fixed points. Let
$U_\bullet$ be a fixed point corresponding to a choice of indices
$i_1, i_2$. Let $\{1,2\} \backslash \{i_1\} = \{q\}$ and let
$\{1,2,3,4\} \backslash \{i_1, i_2\} = \{q_1,q_2\}$. Then the
$\C^*$-equivariant Chern classes of the relevant bundles restricted to
$U_\bullet$ are:
$$
\begin{array}{lll} 
c_{(1,1)}^{\C^*}(\bV/\bU_2)=((w_{q_1}+w_{q_2})^2-w_{q_1} w_{q_2}) t^2 ,&
c_1^{\C^*}(\bU_1^\vee) = (-w_{i_1}) t , &
c_1^{\C^*}((\bU_2/\bU_1)^\vee) = (-w_{i_2}) t , \\
& & \\
\multicolumn{3}{l}{c_3^{\C^*}(TZ) =
  (w_q-w_{i_1})(w_{q_1}-w_{i_2})(w_{q_2}-w_{i_2}) t^3 .}
\end{array} 
$$

If we let $c_1^{\C^*}(\bB) = b t$, then each term of the Bott Residue
Formula applied to \eqref{eq:bott-res-integral} is the following rational number:
$$ \frac{b ((w_{q_1}+w_{q_2})^2-w_{q_1} w_{q_2})}{
  (w_q-w_{i_1})(w_{q_1}-w_{i_2})(w_{q_2}-w_{i_2}) } $$

Now we sum over the fixed point set. The sums below are written according to
the following order for the fixed points:
$$
\begin{array}{lll}
 (\Span{e_1},\Span{e_1,e_2}), & (\Span{e_1},\Span{e_1,e_3}), &
  (\Span{e_1},\Span{e_1,e_4}), \\
 (\Span{e_2},\Span{e_1,e_2}), & (\Span{e_2},\Span{e_2,e_3}), &
 (\Span{e_2},\Span{e_2,e_4})  .
\end{array}
$$

\begin{enumerate}
\item $\int_Z c_1(\bU_1^\vee) c_{(1,1)}(\bV/\bU_2) \cap [Z] = 0 + 0 +
  0 + 1 + 0 + 0 = 1$
\item $\int_Z c_1((\bU_2/\bU_1)^\vee) c_{(1,1)}(\bV/\bU_2) \cap [Z] =
  -3 + 6 + -3 + 0 + 0 + 0 = 0$
\end{enumerate}

Combining these integrals with \eqref{eq:gamma-integral} we have:
$$\gamma_{(2,1),(1,1)} = 3 \cdot 1 + 3 \cdot 0 = 3 .$$
\end{ex} 

\subsection{Explicit Calculations}
\label{subsec:explicit}

We present calculations of the CSM classes of all Schubert cells
contained in the Grassmannian $\GR{3}{6}$. The results were obtained
using formula \eqref{eq:csm-formula-grassmannian}, Theorem
\ref{thm:zel-formula}, and the Bott Residue formula. To make explicit calculations of the classes $\pas \csm(\za)$ when $\dim \za \ge 5$
we use the technique described in \S\ref{subsec:integration} and a computer program written in Sage (see \cite{SAGE}) which enumerates the fixed point set $\za^T$ and keeps track of the $T$-characters.

Schubert cells in $\GR{3}{6}$ are indexed by partitions contained in
$(3,3,3)$. For convenience we label these as
follows:
\[
\begin{array}{lll}
(3, 3, 3)=\alpha_0 & (3, 3)=\beta_0 & (3)=\gamma_0 \\ 
(3, 3, 2)=\alpha_1 & (3, 2)=\beta_1 & (2)=\gamma_1 \\
(3, 2, 2)=\alpha_2 & (2, 2)=\beta_2 & (1)=\gamma_2 \\
(3, 3, 1)=\alpha_3 & (3, 1)=\beta_3 & (0)=\gamma_3 \\
(2, 2, 2)=\alpha_4 & (2, 1)=\beta_4 & \\
(3, 2, 1)=\alpha_5 & (1, 1)=\beta_5 & \\
(2, 2, 1)=\alpha_6 & & \\
(3, 1, 1)=\alpha_7 & & \\
(2, 1, 1)=\alpha_8 & & \\
(1, 1, 1)=\alpha_9 & & 
\end{array}
\]

Table \ref{tab:csm-cell-333} lists the coefficients $[c_{\rm
  SM}(X_\alpha^\circ)]_\beta$ where $\alpha$ indicates the row and
$\beta$ indicates the column. 

One may calculate the CSM class of any Schubert \emph{variety} $\xa$ in
$\GR{3}{6}$ by adding up the rows in this table corresponding to the
cells contained in $\xa$.

The subset of this table consisting of the partitions contained in
$(3,3)$ lists the CSM classes of Schubert cells contained in
$\GR{2}{5}$. These agree with the classes listed in Example 1.2 of
\cite{AM}. The author has also checked that all of the coefficients
appearing in Table \ref{tab:csm-cell-333} agree with the formula given
in \cite{AM}, Theorem 3.10.

\begin{sidewaystable}[!htp] 
\centering
\caption{CSM Classes of Schubert Cells Contained in $\GR{3}{6}$}
\label{tab:csm-cell-333}
\begin{tabular}{c|rrrrrrrrrrrrrrrrrrrr}
 & $\alpha_0$   & $\alpha_1$   & $\alpha_2$   & $\alpha_3$   & $\beta_0$
 & $\alpha_4$
 & $\alpha_5$   & $\beta_1$    & $\alpha_6$   & $\alpha_7$   & $\beta_2$   & $\beta_3$
 & $\alpha_8$   & $\gamma_0$   & $\beta_4$    & $\alpha_9$   & $\gamma_1$   
 & $\beta_5$    & $\gamma_2$   & $\gamma_3$   \\
\hline
$\alpha_0$ &    1&    5&   12&   12&   20&   20&   34&   54&   54&   31&   66&   57&   57&   27&   75&   27&   27&   27&    9&    1\\ 
$\alpha_1$   & . &    1&    4&    4&    8&    8&   15&   27&   27&   17&   39&   34&   34&   18&   51&   18&   21&   21&    8&    1\\ 
$\alpha_2 $  & . & . &    1& . & . &    3&    4&    8&   11&    7&   19&   15&   18&    9&   31&   12&   15&   16&    7&    1\\ 
$\alpha_3 $  & . & . & . &    1&    3& . &    4&   11&    8&    7&   19&   18&   15&   12&   31&    9&   16&   15&    7&    1\\ 
 $  \beta_0$   & . & . & . & . &    1& . & . &    4& . & . &    8&    7& . &    8&   15& . &   12&    9&    6&    1\\ 
$\alpha_4  $ & . & . & . & . & . &    1& . & . &    4& . &    8& . &    7& . &   15&    8&    9&   12&    6&    1\\ 
$\alpha_5  $ & . & . & . & . & . & . &    1&    3&    3&    3&    8&    8&    8&    6&   18&    6&   11&   11&    6&    1\\ 
 $  \beta_1 $  & . & . & . & . & . & . & . &    1& . & . &    3&    3& . &    4&    8& . &    8&    6&    5&    1\\ 
$\alpha_6   $& . & . & . & . & . & . & . & . &    1& . &    3& . &    3& . &    8&    4&    6&    8&    5&    1\\ 
$\alpha_7   $& . & . & . & . & . & . & . & . & . &    1& . &    3&    3&    3&    8&    3&    7&    7&    5&    1\\ 
 $  \beta_2  $ & . & . & . & . & . & . & . & . & . & . &    1& . & . & . &    3& . &    4&    4&    4&    1\\ 
  $ \beta_3  $ & . & . & . & . & . & . & . & . & . & . & . &    1& . &    2&    3& . &    5&    3&    4&    1\\ 
$\alpha_8  $ & . & . & . & . & . & . & . & . & . & . & . & . &    1& . &    3&    2&    3&    5&    4&    1\\ 
 $     \gamma_0 $  & . & . & . & . & . & . & . & . & . & . & . & . & . &    1& . & . &    3& . &    3&    1\\ 
  $ \beta_4  $ & . & . & . & . & . & . & . & . & . & . & . & . & . & . &    1& . &    2&    2&    3&    1\\ 
$\alpha_9  $ & . & . & . & . & . & . & . & . & . & . & . & . & . & . & . &    1& . &    3&    3&    1\\ 
 $     \gamma_1 $  & . & . & . & . & . & . & . & . & . & . & . & . & . & . & . & . &    1& . &    2&    1\\ 
  $ \beta_5$  & . & . & . & . & . & . & . & . & . & . & . & . & . & . & . & . & . &    1&    2&    1\\ 
   $   \gamma_2$   & . & . & . & . & . & . & . & . & . & . & . & . & . & . & . & . & . & . &    1&    1\\ 
    $  \gamma_3$   & . & . & . & . & . & . & . & . & . & . & . & . & . & . & . & . & . & . & . &    1
\end{tabular} 
\end{sidewaystable}

\clearpage

\subsection{Explicit Calculation of Chern-Mather Classes}
\label{subsec:cm-computation}

In this section we use the results of \S\ref{subsec:cm-class} to 
calculate the \tCM classes of Schubert
varieties contained in the Grassmannian $\GR{3}{6}$. To our knowledge
these calculations have not appeared in the literature.

The hypotheses of \S\ref{subsec:cm-class} are satisfied for Schubert varieties in the Grassmannian. The first hypothesis is the existence of small resolutions. This is due to Zelevinsky and we reviewed the construction in 
\S\ref{subsec:zel-reso}.
The second hypothesis is the irreducibility of the characteristic cycle. In \cite{BFL}, Bressler, Finkelberg, and Lunts use small 
resolutions to prove that the characteristic cycle of the 
intersection cohomology sheaf of a Schubert variety in the Grassmannian
is irreducible.

\begin{thm}\label{thm:bfl} (\cite{BFL}, Theorem 0.1)
	Let $\xa \subset \GR{k}{n}$ be a Schubert variety 
	and $\IC{\xa}$ be the corresponding intersection cohomology sheaf. 
	Then \[ CC(\IC{\xa}) = \left[ \overline{T^*_{\xac} \GRb{k}{\V}} \right] .\]
\end{thm}

Thus by Theorem \ref{thm:cm-xa}, the \tCM class of a Schubert variety $\xa
 \subset \GR{k}{n}$ equals \[ \pas \csm(\za) \] where $\pa : \za \to
 \xa$ is a \textit{small} resolution.  
 
 We focus now on the Grassmannian $\GR{3}{6}$. 
To make explicit calculations of the classes $\pas \csm(\za)$ we use the technique described in \S\ref{subsec:integration} and a computer program written in Sage (see \cite{SAGE}) which enumerates the fixed point set $\za^T$ and keeps track of the $T$-characters.
 
 For the table below, Schubert varieties
 contained in $\GR{3}{6}$ are indexed by partitions in $\P{3}{3}$.
 All of the partitions $\alpha \in \P{3}{3}$ except $(3,1)$ and
 $(3,3,1)$ have the property that the resolution \[ \pa : Z_{\alpha,
 w_0} \to \xa ,\] where $w_0$ is the order reversing permutation, is
 small. For $(3,1)$ and $(3,3,1)$ the choice of $s = \id$ gives a
 small resolution. 

 We employ the same ordering and labeling of sub-partitions of $(3,3,3)$
 as in Section \ref{subsec:explicit}. The \tCM classes of Schubert
 varieties in $\GR{3}{6}$ are displayed in Table \ref{tab:cm-333}.
 
 \begin{rmk} 
 The entries in Table \ref{tab:cm-333} are all non-negative. Empirical evidence suggests that the Chern-Mather classes of any Schubert variety are positive (in the sense of \S\ref{sec:pos}). This has been checked by computer for all the Schubert varieties in $\GR{5}{10}$, for instance. The proof of Theorem \ref{thm:weak-pos} also confirms this in the infinite family of special cases considered there.
\end{rmk}

 \begin{sidewaystable}[htp] 
 \centering
 \caption{Chern-Mather Classes of Schubert Varieties in $\GR{3}{6}$}
 \label{tab:cm-333}
 \begin{tabular}{c|rrrrrrrrrrrrrrrrrrrr}
  & $\alpha_0$   & $\alpha_1$   & $\alpha_2$   & $\alpha_3$   & $\beta_0$
  & $\alpha_4$
  & $\alpha_5$   & $\beta_1$    & $\alpha_6$   & $\alpha_7$   & $\beta_2$   & $\beta_3$
  & $\alpha_8$   & $\gamma_1$   & $\beta_4$    & $\alpha_9$   & $\gamma_2$   
  & $\beta_5$    & $\gamma_3$   & $\gamma_4$   \\
 \hline
 $\alpha_0$ &  1 &    6 &   17 &   17 &   32 &   32 &   58 &  108 &  108 &   66 &  174 &  146 &  146 &   90 &  270 &   90 &  150 &  150 &   90 &   20 \\ 
 $\alpha_1$ &  . &    1 &    5 &    5 &   12 &   12 &   24 &   54 &   54 &   36 &  108 &   93 &   93 &   69 &  210 &   69 &  144 &  144 &  108 &   30 \\ 
 $\alpha_2$ &  . &    . &    1 &    . &    . &    4 &    5 &   12 &   19 &   11 &   42 &   30 &   40 &   25 &   98 &   37 &   74 &   82 &   66 &   20 \\ 
 $\alpha_3$ &  . &    . &    . &    1 &    4 &    . &    5 &   19 &   12 &   11 &   42 &   40 &   30 &   37 &   98 &   25 &   82 &   74 &   66 &   20 \\ 
 $\beta_0$  &  . &    . &    . &    . &    1 &    . &    . &    5 &    . &    . &   12 &   11 &    . &   15 &   30 &    . &   35 &   25 &   30 &   10 \\ 
 $\alpha_4$ &  . &    . &    . &    . &    . &    1 &    . &    . &    5 &    . &   12 &    . &   11 &    . &   30 &   15 &   25 &   35 &   30 &   10 \\ 
 $\alpha_5$ &  . &    . &    . &    . &    . &    . &    1 &    4 &    4 &    4 &   15 &   15 &   15 &   17 &   52 &   17 &   54 &   54 &   60 &   24 \\ 
 $\beta_1$  &  . &    . &    . &    . &    . &    . &    . &    1 &    . &    . &    4 &    4 &    . &    7 &   15 &    . &   23 &   17 &   27 &   12 \\ 
 $\alpha_6$ &  . &    . &    . &    . &    . &    . &    . &    . &    1 &    . &    4 &    . &    4 &    . &   15 &    7 &   17 &   23 &   27 &   12 \\ 
 $\alpha_7$ &  . &    . &    . &    . &    . &    . &    . &    . &    . &    1 &    . &    4 &    4 &    6 &   15 &    6 &   21 &   21 &   27 &   12 \\ 
 $\beta_2$  &  . &    . &    . &    . &    . &    . &    . &    . &    . &    . &    1 &    . &    . &    . &    4 &    . &    7 &    7 &   12 &    6 \\ 
 $\beta_3$  &  . &    . &    . &    . &    . &    . &    . &    . &    . &    . &    . &    1 &    . &    3 &    4 &    . &   11 &    6 &   15 &    8 \\ 
 $\alpha_8$ &  . &    . &    . &    . &    . &    . &    . &    . &    . &    . &    . &    . &    1 &    . &    4 &    3 &    6 &   11 &   15 &    8 \\ 
 $\gamma_1$ &  . &    . &    . &    . &    . &    . &    . &    . &    . &    . &    . &    . &    . &    1 &    . &    . &    4 &    . &    6 &    4 \\ 
 $\beta_4$  &  . &    . &    . &    . &    . &    . &    . &    . &    . &    . &    . &    . &    . &    . &    1 &    . &    3 &    3 &    8 &    6 \\ 
 $\alpha_9$ &  . &    . &    . &    . &    . &    . &    . &    . &    . &    . &    . &    . &    . &    . &    . &    1 &    . &    4 &    6 &    4 \\ 
 $\gamma_2$ &  . &    . &    . &    . &    . &    . &    . &    . &    . &    . &    . &    . &    . &    . &    . &    . &    1 &    . &    3 &    3 \\ 
 $\beta_5$  &  . &    . &    . &    . &    . &    . &    . &    . &    . &    . &    . &    . &    . &    . &    . &    . &    . &    1 &    3 &    3 \\ 
 $\gamma_3$ &  . &    . &    . &    . &    . &    . &    . &    . &    . &    . &    . &    . &    . &    . &    . &    . &    . &    . &    1 &    2 \\ 
 $\gamma_4$ &  . &    . &    . &    . &    . &    . &    . &    . &    . &    . &    . &    . &    . &    . &    . &    . &    . &    . &    . &    1
\end{tabular} 
\end{sidewaystable}
\clearpage

\subsection{Direct Pushforward} \label{subsec:direct-pushforward} 

For certain permutations $s$, the class $(\pap)_* \CSM(\zap)$ may be
computed in a more direct fashion. The results in the section may be
used to compute the key pushforward of \eqref{eq:csm-formula-grassmannian}

Let $Z = \zap$ where $s$ is the order reversing permutation and let
$X = \prod_i \GRb{k_i}{V}$ be as in \S~\ref{subsec:csm-reso}. For this choice of $s$,
$Z$ is a subvariety of a partial flag variety for
$\GL_n(\C)$. The directed graph encoding the incidence relations which
define $Z$ is shown in Figure
\ref{fig:direct-pushforward}. 
Recall that for $\alpha = [a_i ; b_i]$ we have set $V^j = V_{a_1 +
  \cdots + a_j}$.
Explicitly, $Z$ is defined to be:
$$ Z = \left\{ U_\bullet \in X \mid V^{m-i} \subset U_i \subset
  U_{i-1} \right\} ,$$ where $U_0 := V$ by convention.

The relations $U_i \subset U_{i-1}$ imply that $Z$ is a subvariety of
a partial flag variety:
$$Z \subset \fl(k_m, k_{m-1}, \ldots, k_1; V) \subset X .$$

\begin{figure}[htbp]
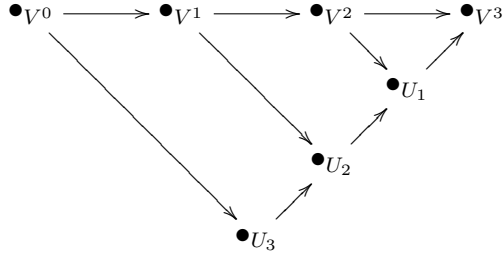

	\centering
	\[ \xygraph{ 
	!{<0cm,0cm>;<1cm,0cm>:<0cm,1cm>::} 
	!{(0,0) }*+{\bullet_{V^0}}="c0" 
	!{(2,0) }*+{\bullet_{V^1}}="c1" 
	!{(4,0) }*+{\bullet_{V^2}}="c2" 
	!{(6,0) }*+{\bullet_{V^3}}="c3" 
	!{(5,-1) }*+{\bullet_{U_1}}="n1" 
	!{(4,-2) }*+{\bullet_{U_2}}="n2" 
	!{(3,-3) }*+{\bullet_{U_3}}="n3" 
	"c0":"c1" "c1":"c2" "c2":"c3"
	"c2":"n1" "n1":"c3"
	"c1":"n2" "n2":"n1"
	"c0":"n3" "n3":"n2"
	} \]
	\caption{Incidence Relation Graph Corresponding to Reverse Peak Order}
	\label{fig:direct-pushforward}
\end{figure}

The strategy of this section is to compute the
pushforward of $\CSM(Z)$ to $A_* \GRb{k_m}{V}$ by representing it
as a cycle in $A_* \fl(k_m, \ldots, k_1; V)$ and using the sequence of
Grassmannian bundles
\[  \fl(k_m, \ldots , k_1; V) \to  \fl(k_m, \ldots, k_2; V) \to
\cdots \to \fl(k_m; V) = \GRb{k_m}{V} .\]

The main task is to understand how to represent the fundamental
class of $Z$ in $A_* \fl(k_m, \ldots, k_1; V)$.
To do this we need the following special case of the Thom-Porteous
formula (\cite{Ful}, Theorem 14.4) as discussed in \cite{Ful}, Example
14.4.16. For the setup, $\bE \to Y$ is a vector bundle over a variety,
$c_\lambda(\bB)$ denotes a Schur determinant in the Chern classes of
$\bB$, and recall that $a \times b$ is a certain rectangular partition.

\begin{prop} \label{prop:thom-porteous} Let $\bA$ be a sub-bundle of
  $\bE$ and let $q, k$ be non-negative integers.  Set $s := \rk \bE - k - \rk
  \bA + q$. Then, the subset $\Omega_q(\bA)$ of $\GRb{k}{\bE}$ defined by
  \[  \{ (y, V) \mid \dim V \cap \bA(y) \ge q \} \] is
  a closed subvariety and its fundamental class is
  \[ [\Omega_q(\bA)] = c_{q \times s}(\bQ - \pi^{-1} \bA) \cap
  [\GRb{k}{\bE}] \]
\end{prop}

Let $l_i := \rk \bU_i - \rk \blW_i$, $r_i := \rk \brW_i - \rk \bU_i$
and $\fl := \fl(k_m, \ldots, k_1; V)$ for brevity. Also, let
$\bU_0 := \bV$ by convention.

\begin{prop} \label{prop:za-fundamental-class}
  \[ [Z] = \left( \prod_{i=1}^m \leftclass{i} \right) \cap [\fl] \]
\end{prop}

\begin{proof}
  We proceed by induction on $m$. Suppose that $m=1$. Then $\fl$ is a
  Grassmannian and $Z \subset \fl$ is a sub-Grassmannian defined by a
  single incidence relation:

\[ Z = \{ U_1 \in \GRb{k_1}{V} \,\, \mid \,\, \lW_1 \subset U_1  \} .\]
  By the Thom-Porteous Formula (Proposition \ref{prop:thom-porteous})
  with $\bA = \blW_1$ and $q = \rk \blW_1 = k_1 - l_1$ we have
  \[ [Z] = c_{(k_1-l_1) \times r_1}(\bQ - \pi^{-1} \bA) \cap [\fl] .\]
  In this case $\bQ = \bV/\bU_1$. To convert to an expression in terms
  of Segre classes of $\bV/\bU_1$ we use the triviality of $\bA$ to
  conclude that $c(\pi^{-1} \bA) = s(\pi^{-1} \bA)~=~1$.  Thus
  \[ [Z] = s_{r_1 \times (k_1-l_1)}(\bV/\bU_1) \cap [\fl] .\] 
  This verifies the proposition in the case $m=1$.

  Suppose that $m>1$. Recall the projection map $\phi$  from
  the proof of Theorem \ref{thm:c-tangent-bundle}
  \[ \phi : \prod_{i=1}^m \GRb{k_i}{V} \to \prod_{i=1}^{m-1}
  \GRb{k_i}{V} .\] The restriction of $\phi$ to $\fl$ is a Grassmannian
  bundle over $\fl^{(m-1)} := \fl(k_{m-1}, k_{m-2}, \ldots, k_{1}; V)$ with $\bE =
  \bU_{m-1}$, $\bS = \bU_m$, and $\bQ = \bU_{m-1}/\bU_m$.
  Recall that $Z^{(m-1)} := \phi(Z)$ is defined by the same type of
  incidence relations as $Z$. Let $\tilde{Z} := \phi^{-1}(Z^{(m-1)})
  \cap \fl$.

Consider the following diagram of varieties   
\[ \xymatrix{
Z \,\, \ar@{^{(}->}@<-0.5ex>[r]^{\iota_Z} \ar@{>>}[d]^{\phi\lvert_{Z}} & 
\,\, \tilde{Z} \,\, \ar@{^{(}->}@<-0.5ex>[r]^{\iota_{\tilde{Z}}} \ar@{>>}[d]^{\phi\lvert_{\tilde{Z}}} & 
\,\, \fl \ar@{>>}[d]^{\phi\lvert_{\fl}} \\
Z^{(m-1)} \,\, \ar@{=}[r] & \,\, Z^{(m-1)} \,\, \ar@{^{(}->}@<-0.5ex>[r] & \,\, \fl^{(m-1)} 
.} \]
Observe that $\tilde{Z}$ is a Grassmannian bundle over $Z^{(m-1)}$ and
moreover, $Z$ is a sub-Grassmannian bundle of $\tilde{Z}$.
Applying the argument of the case $m=1$ here with $\bA = \blW_m$
and $q = k_m - l_m$ we get
\[ [Z] = \leftclass{m} \cap [\tilde{Z}] .\]
By induction we may assume that
\[ [Z^{(m-1)}] = \left( \prod_{i=1}^{m-1} \leftclass{i} \right) \cap [\fl^{(m-1)}].
\]
Since $\tilde{Z} = (\phi\lvert_{\fl})^{-1}(Z^{(m-1)})$, we have
\[ [\tilde{Z}] = \left( \prod_{i=1}^{m-1} \leftclass{i} \right) \cap [\fl] .\] 
The Segre class $\leftclass{m}$ is pulled back from $\fl$ so we can use the
projection formula for $\iota_{\tilde{Z}}$ to conclude
\begin{eqnarray*}
[Z] & = & \leftclass{m} \cdot \left( \prodl_{i=1}^{m-1} \leftclass{i}
\right) \cap [\fl] \\
& = & \left( \prodl_{i=1}^m \leftclass{i}
\right) \cap [\fl] .
\end{eqnarray*}

\end{proof}

The following general result of Fulton and Pragacz gives an explicit
formula for the pushforward of certain classes on a Grassmannian
bundle including those which appear in Theorem
\ref{thm:c-tangent-bundle} and Proposition
\ref{prop:za-fundamental-class}.

Let $\bE$ be a vector bundle on a variety $Y$ and $f : \GRb{k}{\bE}
\to Y$ an associated Grassmannian bundle. Let $\bS$, $\bQ$ denote the
universal sub and quotient vector bundles on $\GRb{k}{\bE}$. As in
Section \ref{subsec:csm-reso}, $s_\lambda(\bB)$ is a Schur determinant
in the Segre classes of $\bB$.

\begin{thm} \label{thm:pushforward-formula} (\cite{FP}, Prop. 4.1)
Let $s = \rk \bS$,
  $q = \rk \bQ$, $\lambda = (\lambda_1, \cdots, \lambda_q)$, and
  $\mu = (\mu_1, \cdots, \mu_s)$ be arbitrary integer sequences. 
Then,
\[ f_*( s_\lambda(\bQ) \cdot s_\mu(\bS) \cap [\GRb{k}{\bE}]) = 
s_{\Lambda}(\bE) \cap [Y] \] 
where
\[ \Lambda = (\lambda_1-s, \ldots, \lambda_q-s,\mu_1, \ldots, \mu_s) .\]
\end{thm}

\begin{rmk}
  If the sequence $\Lambda$ is not a partition, then either there is a
  partition $\nu$ (explicitly determined by $\Lambda$) such that
  $s_\Lambda(\bE) = \pm s_\nu(\bE)$, or else $s_\Lambda(\bE) = 0$. See
  \cite{FP} for details.
\end{rmk}

Let $\psi :=  {\pr_m}\lvert_{\fl} : \fl \to \GRb{k_m}{V}$.
We give an explicit formula for the pushforward under
$\psi$ of classes of the form 
\[ \left( \prod_{i=1}^m s_{\lambda(i)}(\bU_{i-1}/\bU_i) \right) \cap
[\fl] .\]

\begin{prop} \label{prop:direct-pushforward}
\[ \psi_* \left( \prod_{i=1}^m s_{\lambda(i)}(\bU_{i-1}/\bU_i) \cap
[\fl] \right) = s_\Lambda(\bV/\bU_m) \cap [\GRb{k_m}{V}] \]
where $\Lambda$ has a block form $\Lambda = (\Lambda_1, \ldots,
\Lambda_m)$
with
\[ \Lambda_i = \lambda(i) - r_i \times (r_{i+1} + r_{i+2} + \cdots + r_{m}) \]
\end{prop}

	The proof of Proposition \ref{prop:direct-pushforward} 
	proceeds by iteratively applying the Pushforward Formula 
	(Theorem \ref{thm:pushforward-formula}) along with the projection formula.

\begin{ex} \label{ex:direct-pushforward}
	We illustrate in an example with $m=3$.
	Consider the sequence of Grassmannian bundles
	\[ \fl( k_3, k_2, k_1; V ) \overset{\psi_1}{\to} 
	\fl( k_3, k_2; V ) \overset{\psi_2}{\to} 
	\fl(k_3; V) = \GRb{k_3}{V} .\]
	The vector bundles associated to $\psi_i$ are
	$\bE_i := \bV/\bU_{i+1}$, $\bS_i := \bU_i/\bU_{i+1}$, and
	$\bQ_i := \bV/\bU_i$ where, as above, $\bU_0 := \bV$ by convention.
	Note that $\bE_{i} = \bQ_{i+1}$ for $i=1,2$.
	
	We factor $\psi = \psi_2 \circ \psi_1$ and compute:
	\begin{eqnarray*}
		\psi_* \left( \prod_{i=1}^3 s_{\lambda(i)}(\bU_{i-1}/\bU_i) \cap [\fl] \right) & = & 
		(\psi_2 \circ \psi_1)_* \left( \prod_{i=1}^3 s_{\lambda(i)}(\bS_{i-1}) \cap [\fl] \right)\\
		& = & (\psi_2)_* \left( s_{\lambda(3)}(\bS_{2}) \cdot
 		(\psi_1)_* ( s_{\lambda(1)}(\bQ_1) s_{\lambda(2)}(\bS_1) \cap [\fl_1] ) \right) \\
		& = & (\psi_2)_* \left( s_{\lambda(3)}(\bS_{2})
		\cdot s_{( \lambda(1)-(r_1 \times r_2), \lambda(2))}(\bQ_2) \cap [\fl_2] \right) \\
		& = & s_{(\lambda(1)-(r_1 \times (r_2+r_3)), \lambda(2)-(r_2 \times r_3), \lambda(3))}(\bQ_3)  \cap [\fl_3]
	\end{eqnarray*}
	Note that $\bQ_3 = \bV/\bU_3$, so the proposition is verified in the case $m=3$.
\end{ex}

To conclude this section we record a simple positivity statement for
$\csm(Z)$. This will be used in the sequel.

\begin{prop} \label{prop:pos-csm-Z} Let $i : Z \into \fl$ be the
  inclusion. Then, $i_* \csm(Z)$ is a positive class in
  $A_*\fl$. Moreover, each term in the decomposition of
  $i_*\csm(Z)$ given by Corollary \ref{cor:alt-c-tangent-bundle} is
  non-negative.
\end{prop}
\begin{proof}
Corollary \ref{cor:alt-c-tangent-bundle} states
\[ \csm(Z) = \left( \prodl_{i=1}^m \suml_{\mu \subset \lambda \subset
    l_i \times r_i} D^{l_i}_{\lambda, \mu} \, s_\mu((\bU_i/\blW_i)^\vee) \,
  s_{\widetilde{\mathbb{D}}(\lambda)}(\bU_{i-1}/\bU_i) \right) \cap
[Z] .\] 
Using the triviality of
$\blW_i$ and Proposition
\ref{prop:za-fundamental-class} we have
\[ i_* \csm(Z) = \left( \prodl_{i=1}^m \suml_{\mu \subset \lambda
    \subset l_i \times r_i} D^{l_i}_{\lambda, \mu} \,
  s_\mu(\bU_i^\vee) \,
  s_{\widetilde{\mathbb{D}}(\lambda)}(\bU_{i-1}/\bU_i) \, s_{r_i
    \times (k_i-l_i)}(\bU_{i-1}/\bU_i) \right) \cap [\fl] .\]

Recall from Remark \ref{rmk:D-lambda-mu-pos} that the coefficients
$D^N_{\lambda,\mu}$ are non-negative. Thus it suffices to show that each
term is non-negative. It is a standard fact that intersections of
non-negative classes on a (partial) flag variety are
non-negative. This follows from the fact that $\fl$ carries a
transitive algebraic group action. So we are reduced to showing that
$s_\nu(\bU_i^\vee)$ and $s_\nu(\bU_{i-1}/\bU_i)$ are dual to
non-negative classes for all partitions $\nu$. 

For $s_\nu(\bU_i^\vee)$, we observe that
$\bU_i^\vee$ is generated by sections (i.e. it is a quotient
of a trivial bundle). Indeed, there is a short exact sequence on $Z$,
\[ 0 \to \bU_i \to \bV \to \bV/\bU_i \to 0 ,\]
 which dualizes to
 \[ 0 \to (\bV/\bU_i)^\vee \to \bV^\vee \to \bU_i^\vee \to 0 .\] Thus
 $\bU_i^\vee$ is a quotient of $\bV^\vee$ which is trivial.  It is a
 well known fact that Chern classes of generated bundles are dual to
 non-negative classes (cf. \cite{Ful} Example 12.1.7(a) )

For $s_\nu(\bU_{i-1}/\bU_i)$, we use the fact that classes of this
form are dual to the classes of certain Schubert varieties in
$\fl$ and hence are positive (cf. \cite{Fulton:Flags}, \S~8).

\end{proof}

\begin{rmk} \label{rmk:pushforward}
Proposition \ref{prop:direct-pushforward} suffices to calculate the
pushforward of all the classes appearing $[Z] \in A_* \fl$ from Proposition
\ref{prop:za-fundamental-class}. In order to use the above pushforward
results for the CSM class $c(TZ) \cap [Z]$ (regarded in $A_* \fl$) one
first needs to decompose factors of the form $c_\mu( \bU_i/\bV_{m-i}
)$ which appear in $c(TZ)$ (see Theorem \ref{thm:c-tangent-bundle}) in
terms of Chern classes of bundles of the form $\bU_{i-1}/\bU_i$. This
may be done combinatorially in terms of Littlewood-Richardson
coefficients, see \cite{Manivel}, Proposition 1.5.26.
\end{rmk}

\section{Positivity Conjecture}
\label{sec:pos}

In this section we discuss the issue of positivity of CSM classes of
Schubert cells and varieties. The methods developed in previous
sections may be used to prove positivity in several special cases.  We
examine the positivity of classes $\csm(\xa)$ when $\alpha$ has a special
form. The type of partition considered corresponds to a Zelevinsky resolution whose geometry is particularly
simple. For arbitrary $\alpha$ it is possible to prove the positivity
of coefficients in $\csm(\xac)$ with small codimension .

\subsection{Known Results}
\label{subsec:known-pos-results}

In \cite{AM}, Aluffi and Mihalcea state the following 
conjecture.

\begin{conj} \label{conj:pos} (\cite{AM}, Conjecture 1,
  \S~4.1) For all partitions $\alpha \in \P{k}{n-k}$, $\csm(\xac)$
is a positive class in $A_* \GR{k}{n}$.
\end{conj}

Since $\csm(\xac) = [\xa] + \sum_{\beta < \alpha} n_\beta [X_{\beta}]$,
the conjecture is equivalent to the coefficients $n_\beta$
being non-negative. Some cases of the conjecture are proved in \cite{AM} and \cite{Mihalcea}. 
For example the conjecture is known for Schubert cells $\xac \subset \GR{k}{n}$ 
when  $k \le 3$. Also, $n_{\beta}$ is known
to be non-negative when $\beta = (b)$ for $b$ a non-negative integer. In this case,
$\xb$ is isomorphic as a variety to projective space.

Recall that $\csm(\xa) = \sum_{\beta \le \alpha} \csm(\xbc)$. If true,
Conjecture \ref{conj:pos} would imply the following weaker conjecture.
 
\begin{conj} \label{conj:weak-pos} (Weak Positivity Conjecture)
For all partitions $\alpha \in \P{k}{n-k}$, the class $\csm(\xa)$  is positive in $A_* \GR{k}{n}$.
\end{conj}

\subsection{Weak Positivity Conjecture}
\label{subsec:weak_conj}

In this section we prove the weak positivity conjecture for Schubert
varieties whose Zelevinsky resolutions are simple in a certain
sense. This includes the Schubert varieties with isolated
singularities as well as the $\xa \subset \GR{k}{n}$ where $k \le
2$. We are interested in the class of Schubert varieties described below because they represent a simple special case of Schubert varieties having a ``one step'' Zelevinsky resolution.

Consider a partition $\alpha$ such that either $\alpha$, or its conjugate $\tilde{\alpha}$, 
is of the form $(b+p, p^a)$ where $a,b,p > 0$. The diagram of such a partition (or its conjugate) has one long row (of length $b+p$) and all other rows are smaller of the same length ($a$ rows of length $p$). Let $\pap : \zap \to \xa$ be a small
resolution. Proposition \ref{prop:zel-chi-fiber} implies that the
singular locus of $\xa$ is the union of $\xbc$ where $d_{\alpha,
  s}(\beta) > 1$.

\begin{prop} \label{prop:gr-singular-locus}
The following are equivalent:
\begin{enumerate}
\item $\xa \subset \GR{k}{n}$ is a Schubert variety whose singular locus 
is non-empty and coincides with a sub-Grassmannian of $\GR{k}{n}$. 
\item $\alpha$, or its conjugate, is of the form $( b+p, p^a)$ where 
$a,b,p > 0$.
\end{enumerate}
\end{prop}

\begin{proof}
  The singular locus of $\xa$ is $B$-stable so the hypothesis implies
  that the singular locus is equal to $\xb$ where $\beta$ is a
  rectangular sub-partition of $\alpha$. In light of Proposition
  \ref{prop:zel-chi-fiber}, the depth vector of $\beta$ relative to
  $\alpha$ must be $(0,1,0)$ (i.e. $\alpha$ must have a single
  interior depression whose height above $\beta$ is $1$). Thus $\alpha$ has the form stated in (2). 
  
  To see that (2) implies (1), use Proposition \ref{prop:zel-chi-fiber} and the fact that the largest sub-partition $\beta$ of $\alpha$ with non-zero depth vector is the rectangle $\beta = ((p-1)^{a})$.
\end{proof}

\begin{rmk} The Schubert varieties described in the proposition
  include those with isolated singularities. This is the case when the
  sub-Grassmannian is a point and $\alpha = (b+1, 1, \ldots, 1)$ for
  some $b > 0$. Also included are the singular Schubert varieties of
  $\GR{2}{n}$.
\end{rmk}

The following theorem provides evidence for the weak positivity conjecture. We
note that the special case included here of Schubert varieties in $\GR{2}{n}$ is implied by the corresponding Schubert cell positivity result of Aluffi-Mihalcea (see \cite{AM}, \S4.3-4.4).

\begin{thm} \label{thm:weak-pos}
Let $\alpha$ (or its conjugate) be of the form $(b+p, p^a)$ for integers $b,p,a > 0$. Then, $\csm(\xa)$ is positive.
\end{thm}

The remainder of this section is devoted to proving the theorem.

By \emph{mutatis-mutandis} (see \cite{AM}, \S1), the set of coefficients occuring
in $\csm(\xa)$ is the same as the set for
$\csm(X_{\tilde{\alpha}})$. Hence positivity of $\csm(\xa)$ implies
positivity of $\csm(X_{\tilde{\alpha}})$ and vice-versa. Without loss
of generality, let $\alpha = (b+p, p^a)$ for some $a,b,p > 0$. Let
$X_\beta$ be the singular locus of $\xa$, i.e.  $\beta = ((p-1)^{a+1})$.

As in \S\ref{subsec:direct-pushforward} we consider Zelevinsky resolutions which
correspond to the order reversing permutation. Let $\pa : \za \to \xa$
denote this resolution of $\xa$. 

\begin{lem}
\label{lem:weak-pos-formula}
$$\csm(\xa) = \pas \csm(\za) - b \cdot \csm(X_{\beta}) .$$
\end{lem}
\begin{proof}
This is a straightforward application of functoriality of
$c_*$. Indeed, by Theorem \ref{prop:zel-chi-fiber}, $\pa$ is a bijection over
$\xa \backslash \xb$ and has fiber isomorphic to $\CP^b$ over
$\xb$. Thus 
$$\pas \one_{\za} = \one_{\xa} + (\Chi(\CP^b)-1) \cdot \one_{\xb} = 
\one_{\xa} + b \cdot \one_{\xb}$$
and hence by functoriality,
$$\csm(\xa) = \pas \csm(\za) - b \cdot \csm(\xb) .$$
\end{proof}

Let
$\alpha'$ denote the partition obtained by removing the rightmost peak
from $\alpha$. In peak notation we have $\alpha = [a, 1; p,
b]$ and $\alpha' = [a+1; p]$.

Recall the setup from \S~\ref{subsec:direct-pushforward}. We have $i :
\za \into \fl$, where $\fl$ is a certain partial flag variety. Let
$\pi$ denote the projection $\pi : \fl \to \GR{k}{n}$. Also recall 
the expression of $\csm(\za)$ from Theorem
\ref{thm:c-tangent-bundle}:
\begin{equation} \label{eq:weak-pos-ctza} 
\csm(\za) = \left( c( \bU_2^\vee \otimes (\bU_1/\bU_2) ) \cdot c(
  (\bU_1/\bV^{a+p+1})^\vee \otimes (\bV/\bU_1) ) \right) \cap [\za]
\end{equation}

The following lemma shows that the pushforward of a certain term
occurring in $i_* \csm(\za)$ to $A_* \GR{k}{n}$ is the CSM class of
$X_{\alpha'}$.

\begin{lem} \label{lem:weak-pos-pushforward} 
The pushforward 
\begin{equation}
  \label{eq:weak-pos-pushf-term} \pi_* \left( c( \bU_2^\vee \otimes
    (\bU_1/\bU_2)) \cdot s_{\tilde{(b)}}( \bV/\bU_1 ) \cdot s_{b
      \times (a+p)}( \bV/\bU_1 ) \cap [\fl] \right) 
\end{equation}
equals $\csm(X_{\alpha'})$ in $A_* \GR{k}{n}$.
\end{lem}

Note that the first factor of \ref{eq:weak-pos-pushf-term} is one of
the factors in the decomposition of $c(T\za)$ given by Theorem
\ref{thm:c-tangent-bundle}. The second two factors come from the
other part of $c(T\za)$ and the class $i_* [\za]$, respectively.
\begin{proof}[Proof (of Lemma \ref{lem:weak-pos-pushforward})]
First note that since $\rk \bV/\bU_1 = b$ we may combine the last two
factors in the left hand side expression as follows:
$$  s_{\tilde{(b)}}(
  \bV/\bU_1 ) \cdot s_{b \times (a+p)}( \bV/\bU_1 ) = 
s_{b \times (a+p+1)}(\bV/\bU_1) .$$
Using Lemma  \ref{lem:schur-ctz} to expand the first factor of 
\eqref{eq:weak-pos-pushf-term} we have:
$$\pi_* \left[ \left( \sum_{\mu \le \lambda \le (a+1) \times p} 
D^{a+1}_{\lambda, \mu} s_\mu(\bU_2^\vee) \cdot
s_{\TDual{\lambda}}(\bU_1/\bU_2) \cdot
s_{b \times (a+p+1)}(\bV/\bU_1) \right) \cap [\fl] \right].$$

We now apply the pushforward formula of Theorem
\ref{thm:pushforward-formula}. Note that in the special case we are
working with, $\pi : \fl \to \GR{k}{n}$ is a Grassmannian bundle with
associated vector bundles $\bE = \bV/\bU_2$, $\bS = \bU_1/\bU_2$, and
$\bQ = \bV/\bU_1$. Applying the pushforward and identifying $\bU_2 =
\bU$, the universal sub-bundle on $\GR{k}{n}$,
\eqref{eq:weak-pos-pushf-term} becomes:
\begin{equation} \label{eq:weak-pos-pushf-result}
\left( \sum_{\mu \le \lambda \le (a+1) \times p} 
D^{a+1}_{\lambda, \mu} s_\mu(\bU^\vee) \cdot s_{(b \times (a+1),
  \TDual(\lambda))}(\bV/\bU) \right) \cap [\GR{k}{n}] .
\end{equation}
Observe that $\lambda \le (a+1) \times p$ implies $\TDual(\lambda) \le
p \times (a+1)$. Thus, the sequence $( b \times (a+1), \TDual(\lambda)
)$ is in fact a partition. 

Now we compare the expression \eqref{eq:weak-pos-pushf-result} to the
class $\csm(X_{\alpha'})$. To compute $\csm(X_{\alpha'})$ note that
$\alpha'$ is rectangular and thus $X_{\alpha'}$ is smooth and
isomorphic to a Grassmannian. Let $\bU$ denote the universal
sub-bundle on $\GR{k}{n}$. In the embedding $X_{\alpha'} \into
\GR{k}{n}$ we identify the tangent bundle of $X_{\alpha'}$ with
$\bU^\vee \otimes (\bV^{a+p+1}/\bU)$ and by the Thom-Porteous formula,
$[X_{\alpha'}] = s_{b \times (a+1)}(\bV/\bU) \cap
[\GR{k}{n}]$. Expanding $c(TX_{\alpha'})$ we have:
$$ \csm(X_{\alpha'}) = \left( \sum_{\nu \le \eta \le (a+1) \times p} D^{a+1}_{\eta,
  \nu} s_\nu(\bU^\vee) \cdot s_{\TDual(\eta)}(\bV^{a+p+1}/\bU) \cdot 
s_{b \times (a+1)}(\bV/\bU) \right) \cap [\GR{k}{n}] .$$
Combining the last two factors (both are equal to Schur determinants in
the Segre classes of $\bV/\bU$ on
$\GR{k}{n}$) as in the first paragraph, the result follows:
\end{proof}

Now we can prove Theorem \ref{thm:weak-pos}.

\begin{proof}[Proof (of Threorem \ref{thm:weak-pos})]

  As noted above, the term which is pushed forward in Lemma
  \ref{lem:weak-pos-pushforward} is (up to an integer coefficient) a
  term occurring in the decomposition of $\csm(\za)$. Let 
$i_* \csm(\za) = m \cdot T + T'$, where $m \in \mathbb{Z}_{\ge 0}$,
$T$ is the term from the lemma  and $T'$ makes up the difference.

Proposition \ref{prop:pos-csm-Z} implies that $T'$ is a positive class in 
$A_* \fl$. Hence it is positive when pushed forward to $A_*
  \GR{k}{n}$. Using Lemma \ref{lem:weak-pos-formula} we have:
\begin{align*}
\csm(\xa) & = \pas \csm(\za) - b \cdot \csm(\xb) \\
          & = \pi_* i_* \csm(\za) - b \cdot \csm(\xb) \\
          & = m \cdot \pi_* T + \pi_* T' - b \cdot \csm(\xb)
\end{align*}
The idea is to show that $m \cdot \pi_* T -  b \cdot \csm(\xb)$ is
in fact positive. Since $\pi_* T' > 0$, the result follows.

We claim that $\pi_* T - \csm(\xb) > 0$ and $m > b$. From Lemma
\ref{lem:weak-pos-pushforward}, $\pi_* T = \csm(X_{\alpha'})$. Note
that both $\alpha'$ and $\beta$ are rectangular, and that $\beta <
\alpha'$, the difference being a single column. Then the fact that 
$\csm(X_{\alpha'}) - \csm(\xb) > 0$ is a routine exercise in comparing
the $D^N_{\lambda, \mu}$ coefficients. This is left to the reader.

The coefficient $m$ comes from occurrences of the term $t =
s_{\tilde{(b)}}(\bV/\bU_1)$ in the second factor of
$c(T\za)$ (from \eqref{eq:weak-pos-ctza}) which we expand:
$$ c( (\bU_1/\bV^{a+p+1})^\vee \otimes (\bV/\bU_1)) = \sum_{\mu \le
  \lambda \le 1 \times b} D^1_{\lambda, \mu}
s_\mu((\bU_1/\bV^{a+p+1})^\vee) s_{\TDual(\lambda)}(\bV/\bU_1) .$$ Observe that
the term $t$ occurs in this expression precisely when $\mu = \lambda$. For each such pair $(\mu, \lambda)$, $D^1_{\lambda,
  \mu} = 1$, thus we have
$$m = \sum_{\mu = \lambda \le 1 \times b} D^1_{\lambda, \mu} =
\sum_{i=0}^b 1 = b+1 .$$
Hence $m > b$ and the proof is concluded.


\end{proof}

\subsection{Coefficients with Small Codimension}
\label{subsec:codimension-one}

The methods developed in this paper can be used to prove that the coefficient of $\xb$ in $\csm(\xac)$ is positive when $\xac$ is arbitrary and $\xb$ has small codimension in $\xa$. We sketch the idea  of this proof for the codimension one case and give a combinatorial interpretation of the coefficient. We remark that the positivity results for small codimension may also be obtained using the formulas of \cite{AM}. 

Consider an arbitrary partition $\alpha$ and the class $\csm(\xac) = [\xa] +\sum_{\beta < \alpha} n_\beta [\xb]$. We want to calculate the coefficients $n_\beta$ when $\beta < \alpha$ and $\abs{\beta} = \abs{\alpha} - 1$. The first step is to expand $n_\beta$ in terms of (coefficients of) CSM classes of resolutions. Let $\pa: \za \to \xa$ and $\pb : \zb \to \xb$ be Zelevinsky resolutions as in \S\ref{subsec:weak_conj}. Using \ref{eq:csm-formula-grassmannian} and the fact that $\xa$ is smooth at points in $\xbc$ (see last paragraph of \S\ref{subsec:csm-by-resolution}), 
we have 
\[ n_\beta = [ \pas \csm(\za) ]_\beta - [ \pbs \csm(\zb) ]_\beta =
[ \pas \csm(\za) ]_\beta - 1 .\]
It suffices then to show that $[ \pas \csm(\za) ]_\beta > 1$. For reasons of dimension, we need only consider the contribution of $\pas( c_1(T\za) \cap [\za] )$. Using Theorem \ref{thm:c-tangent-bundle} this is a straightforward calculation: $c_1(T\za)$ is a sum of terms: \[ c_1(
(\bU_j/\blW_j)^\vee \otimes (\bU_{j-1}/\bU_j) ) = r_j \cdot c_1(
(\bU_j/\blW_j)^\vee ) + l_j \cdot c_1( \bU_{j-1}/\bU_j ) \] where the coefficients $r_j, l_j$ are non-negative. One can then calculate the contribution to $n_\beta$ of $\pas (c_1(T\za) \cap [\za])$ explicitly using 
Theorem \ref{thm:pushforward-formula}. To state the result precisely we need some notation. Let $\alpha = [a_1, \ldots, a_m ; b_0, \ldots, b_{m-1}]$ in peak notation. For $i = 1, \ldots, m$ let $\beta(i)$ denote the partition obtained by removing the top-most box from the $i$-th peak of $\alpha$ (see Figure \ref{fig:peak-diagram}).

 \begin{thm} \label{thm:codim-1-formula} With notation as above,
   \[ [ \pas \csm(\za) ]_{\beta(i)} = ( a_i + a_{i+1} + \cdots + a_m )
   + (b_0 + b_1 + \cdots + b_{i-1} ) > 1.\] In particular,
   \begin{equation} \label{eq:codim-one-strict-pos} [ \csm(\xac)
     ]_{\beta(i)} > 0.
   \end{equation}	
 \end{thm}

 \begin{rmk}
   The formula of Theorem \ref{thm:codim-1-formula} has a combinatorial
   interpretation.  Consider the partition diagram for $\alpha$.  The
   integer \[ ( a_i + a_{i+1} + \cdots + a_m ) + (b_0 + b_1 + \cdots +
   b_{i-1} ) \] is 1 more than the number of boxes in the ``anti-hook''
   whose corner is the box in the partition diagram of $\alpha$
   which is removed to produce $\beta(i)$. By the theorem, the number of boxes in the anti-hook
   is precisely the coefficient $[\csm(\xac)]_\beta$. See Example
   \ref{ex:hook} below for illustration.
 \end{rmk}

 \begin{ex} \label{ex:hook} 
   Suppose $\alpha = (5,5,2,1) = [1,1,2; 1,1,3]$. Let $\beta = \beta(2)$. The
   coefficient of $[\xb]$ in $\csm(\xac)$ is $((1+2)+(1+1))-1 = 4$. The anti-hook is
   depicted below in Figure \ref{fig:hook}.
   \begin{figure}[hbtp]
     \centering
     \includegraphics[scale=0.75]{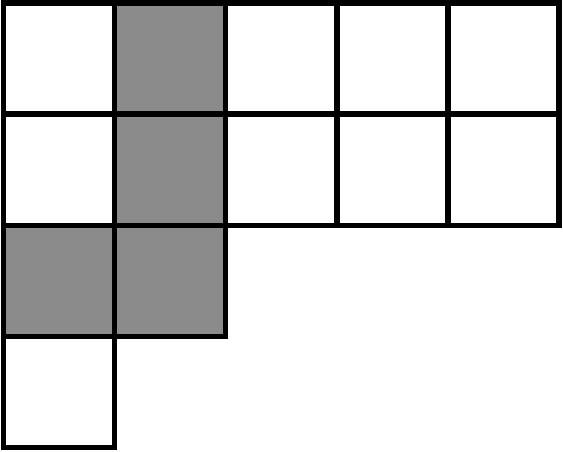}
     \caption{Anti-hook length counts codimension 1 coefficients}
     \label{fig:hook}
   \end{figure}
 \end{ex}


\newcommand{\bibnumfmt}[1]{[#1]} 
\bibliographystyle{bfjthesis}
\bibliography{thesis}          

\end{document}